\documentclass[leqno]{article}

\usepackage{amssymb,amsfonts,amsthm,amsmath} 
\usepackage{a4wide}
\usepackage{graphicx}
\usepackage[usenames,dvipsnames,svgnames,table]{xcolor}
\usepackage[colorlinks=true, pdfstartview=FitV,linkcolor=ForestGreen,citecolor=ForestGreen, urlcolor=blue]{hyperref}
\usepackage[utf8]{inputenc}
\usepackage[T1]{fontenc}
\usepackage[margin=1in]{geometry} 
\usepackage{esint} 
\usepackage{enumitem}
\usepackage{epsfig}

\newcommand{\R}{\mathbb{R}}

\newcommand{\eps}{\varepsilon}

\newtheorem{theo}{\bf Theorem}[section]
\newtheorem{lem}[theo]{\bf Lemma}
\newtheorem{pro}[theo]{\bf Proposition}
\newtheorem{cor}[theo]{\bf Corollary}
\newtheorem{defi}[theo]{\bf Definition}
\theoremstyle{remark}
\newtheorem{rem}[theo]{Remark}
\newtheorem{counterexample}[theo]{Counter-example}
\newtheorem{ex}[theo]{Example}

\numberwithin{equation}{section}

\begin{document}

\title{\textbf{Non-convex coercive Hamilton-Jacobi equations: \\
  Guerand's relaxation revisited}}
\author{Nicolas Forcadel, Cyril Imbert et Régis Monneau}
\date{\today}
  \maketitle

  \begin{abstract}
    This work is concerned with Hamilton-Jacobi equations of evolution type posed
    in domains and supplemented with boundary conditions. Hamiltonians are coercive
    but are neither convex nor quasi-convex. We analyse boundary conditions when
    understood in the sense of viscosity solutions. This analysis is based
    on the study of boundary conditions of evolution type. More precisely,
    we  give a new formula for the relaxed boundary conditions derived by J. Guerand
    ({\it J. Differ. Equations}, 2017). This new point of view unveils a connection between
    the relaxation operator and the classical
    Godunov flux from the theory of conservation laws. We apply our methods to
     two classical boundary value problems.
    It is shown that the relaxed Neumann boundary  condition is expressed in terms of Godunov's flux while
    the relaxed Dirichlet boundary condition reduces to an obstacle problem at the boundary
    associated with the lower non-increasing envelope of the Hamiltonian. 
\end{abstract}


\section{Introduction}

When a partial differential equation is posed in a domain, the 
boundary condition may be in conflict with the equation. This typically happens
when characteristics reach the boundary. More specifically, such a phenomenon is observed for 
evolutive Hamilton-Jacobi (HJ) equations. A classical way to handle this discrepancy is to impose either
the boundary condition or the equation at the boundary, both in terms of viscosity solutions.
Such viscosity solutions are called \emph{weak}. 

The second and third authors studied HJ equations on networks \cite{zbMATH06713740} for
coercive and convex Hamiltonians. The equations are supplemented with conditions at junctions (vertices). 
When these conditions are compatible with the maximum principle, it is easy to construct weak viscosity solutions by Perron's method.
In this previous work, the authors proved that these weak solutions satisfy \emph{other} boundary (junction) conditions in a
strong sense. The family of these \emph{relaxed} boundary conditions is completely characterized by a real parameter, the \emph{flux limiter}. 

When Hamiltonians are coercive but not necessarily convex, J. Guerand has shown in the mono-dimen-sional setting  \cite{zbMATH06731793} that
 it is also possible to characterize relaxed boundary conditions associated with general boundary
conditions compatible with the maximum principle. In this case, the family of relaxed boundary conditions is much richer and
is characterized by a family of \emph{limiter} points. 
With this tool at hand, she established a comparison principle for general boundary conditions in this framework.

In this work, we are interested in the multi-dimensional case and we treat both dynamic, Neumann
and Dirichlet boundary conditions. As far as dynamic boundary conditions are concerned,
we give a new formula for the relaxed boundary conditions obtained by J. Guerand. It is easily derived
from the definition of weak viscosity  solutions. We also exhibit a deeply rooted connection between the relaxed
dynamic boundary condition and Godunov's flux for conservation laws. This classical numerical flux also appears in the formula for the relaxed Neumann boundary condition.
As far as the Dirichlet boundary condition is concerned, relaxation yields an obstacle problem at the boundary.


\subsection{Coercive Hamilton-Jacobi equations posed on domains}

In this article, we are interested in the study of Hamilton-Jacobi (HJ) equations of evolution type posed in a $C^1$ domain
$\Omega$ of $\R^d$ and supplemented with boundary conditions. We shall see that the study of boundary conditions of
evolution type 
\begin{equation}\label{eq:HJ-bc}
\begin{cases}
u_t + H(t,x,Du)=0, &  t>0, x \in \Omega,\\
u_t + F_0(t,x,Du)=0, & t>0, x \in \partial \Omega
\end{cases}
\end{equation}
is suprisingly fruitful in the understanding of general boundary conditions that are compatible with the maximum principle.
In particular, it gives a new insight on the classical inhomogeneous Neumann problem,
\begin{equation}\label{eq:HJ-neu}
\begin{cases}
  u_t + H(t,x,Du)=0, &  t>0, x \in \Omega,\\
\frac{\partial u}{\partial n}+h(t,x)=0, & t>0, x \in \partial \Omega
\end{cases}
\end{equation}
and on the Dirichlet problem,
\begin{equation}\label{eq:HJ-dir}
\begin{cases}
u_t + H(t,x,Du)=0, &  t>0, x \in \Omega,\\
u=g(t,x), & t>0, x \in \partial \Omega.
\end{cases}
\end{equation}

In~\eqref{eq:HJ-neu} and \eqref{eq:HJ-dir}, the
functions  $h,g \colon (0, +\infty) \times \partial \Omega \to \R$ are continuous and $\frac{\partial u}{\partial n}$ denotes
the normal derivative associated with  the outward unit normal vector field $n\colon \partial \Omega \to \R^d$. 
Throughout this work, we make the following assumption,
\begin{equation}\label{a:HandF0}
  H,F_0\colon (0+\infty)\times \Omega \times \R^d \to \R \text{ are continuous, }
  \partial \Omega \in C^1, F_0 \text{ is non-decreasing in } \frac{\partial u}{\partial n}
    \text{ and } H \text{ is coercive.}
\end{equation}
The coercivity of the Hamiltonian $H$ means that \(H(t,x,p)\) tends to $+\infty$ as \(|p|\to +\infty .\)
We assume very often that $F_0$ is \emph{semi-coercive}, that is to say  it satisfies the following condition,
\begin{equation}
  \label{eq:semi-coercive}
  F_0(t,x,p)\to +\infty \quad \text{as}\quad p\cdot n(x) \to +\infty.
\end{equation}
It is also useful to deal with cases in which this condition on the function $F_0$ is not satisfied.
It is for instance interesting to consider constant $F_0$ functions.  

\paragraph{Weak and strong viscosity solutions for HJ equations posed in domains.}
It is known that classical solutions to Hamilton-Jacobi equations do not exist in general while viscosity solutions
are easily constructed by Perron's method \cite{zbMATH04142546}. As far as boundary conditions are concerned, because characteristics can exit the domain,
boundary conditions are generally ``lost'' for Hamilton-Jacobi equations. As first observed by H. Ishii \cite{zbMATH01414249}, it is useful
to consider viscosity solutions that satisfy either the equation or the boundary condition on $\partial \Omega$.
Such viscosity solutions are called weak. They are easily constructed thanks to Perron's method \cite{zbMATH04142546}. On the contrary, if the boundary condition is always satisfied on
$\partial \Omega$, we say that viscosity solutions are strong. It is usually easier to prove uniqueness of strong viscosity solutions
than to prove uniqueness of weak ones. 

In this article, it is proved that weak viscosity
solutions associated with \eqref{eq:HJ-bc} or  \eqref{eq:HJ-neu} or \eqref{eq:HJ-dir}  are strong viscosity solutions
for \emph{other} boundary conditions that we identify. We start with \eqref{eq:HJ-bc}.

\begin{theo}[Relaxed boundary condition -- \cite{zbMATH06713740,MR3690310,zbMATH06731793}] \label{t:main}
  Assume that $H,F_0\colon \R^d \to \R$ are continuous, $p \mapsto F_0(p)$ is non-decreasing with respect to $p \cdot n$,
  $H$ is coercive and $F_0$ is semi-coercive (in the sense of \eqref{eq:semi-coercive}). 

Then, there exists a continuous semi-coercive function $  \mathfrak R F_0 \colon \R^d \to \R$ such that
a function $u \colon (0,+\infty) \times \Omega$ is a  weak viscosity solution of \eqref{eq:HJ-bc}
if and only if it is a strong   viscosity solution of
\begin{equation*}
\begin{cases}
u_t + H(t,x,Du)=0, & t>0, x\in \Omega,\\
u_t + \mathfrak R F_0(t,x,Du)=0, & t>0, x \in \partial \Omega.
\end{cases}
\end{equation*}
\end{theo}
If $F_0$ is not semi-coercive, the result still holds true if $u$ satisfies a weak continuity assumption at the boundary
of $\partial \Omega$: for all $x \in \partial \Omega$ and $t >0$,
\[
u^*(t,x)=\limsup_{(s,y)\to (t,x), y \in \Omega} u(s,y),
\]
see Theorem~\ref{th:weak-strong} in Section~\ref{s3}.

The application mapping $F_0$ to $\mathfrak{R} F_0$ is referred to as the \emph{relaxation operator.}
Theorem~\ref{t:main} was proved by the second and the third authors \cite{zbMATH06713740,MR3690310} under
the additional assumption that the Hamiltonian $H$ is convex and $\Omega$ is a half-space. In this case, the relaxation operator takes a very simple form
since $\mathfrak R F_0$ is the maximum of a constant $A_0$ (depending on $H$ and $F_0$) and the the lower non-increasing envelope of $H$ given by the formula 
\( H_-(t,x,p):=\inf_{\rho\le 0} H(t,x,p-\rho n(x)).\) When Hamiltonians are coercive but are not convex, J. Guerand \cite{zbMATH06731793} identified the relaxation operator 
in the monodimensional setting. The formula she obtained for $\mathfrak R F_0$ is referred in this article as Guerand's operator and is denoted by $\mathfrak J F_0$;
it is given in Definition~\ref{defi:guerand-relaxed}.

\subsection{A new formula for the relaxation operator}

The first main result of this article is a new formula for the relaxation operator. We first present
it in the mono-dimensional setting for the sake of clarity. 

\subsubsection{The homogeneous mono-dimensional case}

To simplify the presentation, we assume here that $\Omega =(0,+\infty)$ and $H$ and $F_0$ don't depend on $(t,x)$. Let $u$ be a (continuous) weak viscosity solution to \eqref{eq:HJ-bc}. As explained above,
this means that either the equation or the boundary condition is satisfied in the sense of viscosity solutions; see  Definition~\ref{defi:weak-visc-sol} for a precise definition. 
Consequently, if $\varphi$ is a test function touching $u$ from above at $P_0=(t_0,0)$, 
then 
\[\varphi_t+H(\varphi_x)\le 0\quad \text{or}\quad \varphi_t + F_0(\varphi_x)\le 0 \quad \text{ at } P_0\]
or equivalently $\varphi_t+(F_0 \wedge H)(\varphi_x)\le 0$ at $P_0$ where $F_0 \wedge H$ denotes the minimum of $F_0$ and $H$.
Keeping in mind the discussion above about weak and strong viscosity solutions, we obtained a first boundary condition that
is satisfied in a strong sense.

We next derive a more precise one. For any $q\ge \varphi_x (P_0)=:p$, the test function
$\tilde \varphi(t,x)=\varphi(t,x)+ (q-p)\cdot x \ge u(t,x)$ also touches $u$ at $P_0$ from above.
In particular, we also have \(\varphi_t+(F_0 \wedge H)(q)\le 0\) at \(P_0.\)
We conclude that
\[\varphi_t+ \underline R F_0 (\varphi_x) \le 0 \quad \text{ at } P_0\] 
where the operator $\underline R$ is defined by
\begin{equation}\label{eq::r16}
\underline R F_0(p):=\sup_{q\ge p} \ (F_0 \wedge H) (q) .
\end{equation}

Similarly if a test function $\varphi$  touches  a weak solution $u$ of \eqref{eq:HJ-bc} from below at $P_0$, we get
\[\varphi_t+ \overline R F_0 (\varphi_x) \ge 0 \quad \text{ at } P_0\]
where the operator $\overline R$ is defined by
\begin{equation}\label{eq::r17}
\overline R F_0(p):=\inf_{q\le p}\ (F_0 \vee H)(q).
\end{equation}
We refer the reader to Figure \ref{F3} for a representation of the effects of $\underline R$ and $\overline R$ on $F_0$.

We next remark that $\overline R F_0 = \underline R F_0=F_0$ in $\{ F_0 = H\}$ (see Remark~\ref{rem:relax-prop-1} below). 
We  define the relaxation operator $\mathfrak R F_0$ as follows,
\begin{equation}\label{eq::r12}
  \mathfrak R F_0 =
  \begin{cases}
    \underline R F_0 & \text{ in } \{ F_0 \ge H \}, \\
    \overline R F_0 & \text{ in } \{ F_0 \le H\} .
  \end{cases}
\end{equation}
 We refer the reader to Figure \ref{F3} for a representation of the effects of $\mathfrak R$ on $F_0$.

 \begin{figure}[!ht]
\centering
\def\svgwidth{.3\textwidth}
\begingroup%
  \makeatletter%
  \providecommand\color[2][]{%
    \errmessage{(Inkscape) Color is used for the text in Inkscape, but the package 'color.sty' is not loaded}%
    \renewcommand\color[2][]{}%
  }%
  \providecommand\transparent[1]{%
    \errmessage{(Inkscape) Transparency is used (non-zero) for the text in Inkscape, but the package 'transparent.sty' is not loaded}%
    \renewcommand\transparent[1]{}%
  }%
  \providecommand\rotatebox[2]{#2}%
  \newcommand*\fsize{\dimexpr\f@size pt\relax}%
  \newcommand*\lineheight[1]{\fontsize{\fsize}{#1\fsize}\selectfont}%
  \ifx\svgwidth\undefined%
    \setlength{\unitlength}{595.27559055bp}%
    \ifx\svgscale\undefined%
      \relax%
    \else%
      \setlength{\unitlength}{\unitlength * \real{\svgscale}}%
    \fi%
  \else%
    \setlength{\unitlength}{\svgwidth}%
  \fi%
  \global\let\svgwidth\undefined%
  \global\let\svgscale\undefined%
  \makeatother%
  \begin{picture}(1,1.41428571)%
    \lineheight{1}%
    \setlength\tabcolsep{0pt}%
    \put(0,0){\includegraphics[width=\unitlength,page=1]{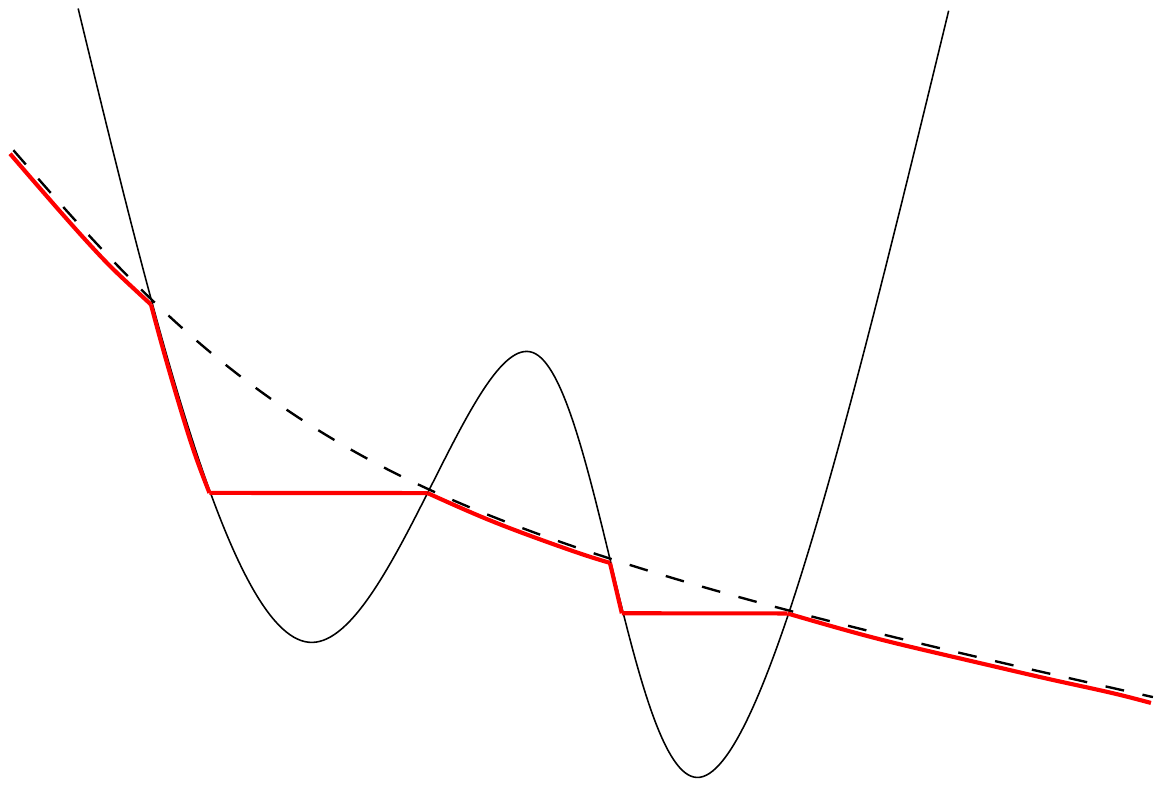}}%
    \put(0.67225763,0.95261478){\color[rgb]{0,0,0}\makebox(0,0)[lt]{\lineheight{1.25}\smash{\begin{tabular}[t]{l}$H$\end{tabular}}}}%
    \put(0.23938491,0.88691897){\color[rgb]{0,0,0}\makebox(0,0)[lt]{\lineheight{1.25}\smash{\begin{tabular}[t]{l}$F_0$\end{tabular}}}}%
    \put(-0.00809949,0.4432469){\color[rgb]{1,0,0}\makebox(0,0)[lt]{\lineheight{1.25}\smash{\begin{tabular}[t]{l}$\underline R F_0$\end{tabular}}}}%
    \put(0,0){\includegraphics[width=\unitlength,page=2]{underlineRF02.pdf}}%
  \end{picture}%
\endgroup%

\def\svgwidth{.3\textwidth}
\begingroup%
  \makeatletter%
  \providecommand\color[2][]{%
    \errmessage{(Inkscape) Color is used for the text in Inkscape, but the package 'color.sty' is not loaded}%
    \renewcommand\color[2][]{}%
  }%
  \providecommand\transparent[1]{%
    \errmessage{(Inkscape) Transparency is used (non-zero) for the text in Inkscape, but the package 'transparent.sty' is not loaded}%
    \renewcommand\transparent[1]{}%
  }%
  \providecommand\rotatebox[2]{#2}%
  \newcommand*\fsize{\dimexpr\f@size pt\relax}%
  \newcommand*\lineheight[1]{\fontsize{\fsize}{#1\fsize}\selectfont}%
  \ifx\svgwidth\undefined%
    \setlength{\unitlength}{595.27559055bp}%
    \ifx\svgscale\undefined%
      \relax%
    \else%
      \setlength{\unitlength}{\unitlength * \real{\svgscale}}%
    \fi%
  \else%
    \setlength{\unitlength}{\svgwidth}%
  \fi%
  \global\let\svgwidth\undefined%
  \global\let\svgscale\undefined%
  \makeatother%
  \begin{picture}(1,1.41428571)%
    \lineheight{1}%
    \setlength\tabcolsep{0pt}%
    \put(0,0){\includegraphics[width=\unitlength,page=1]{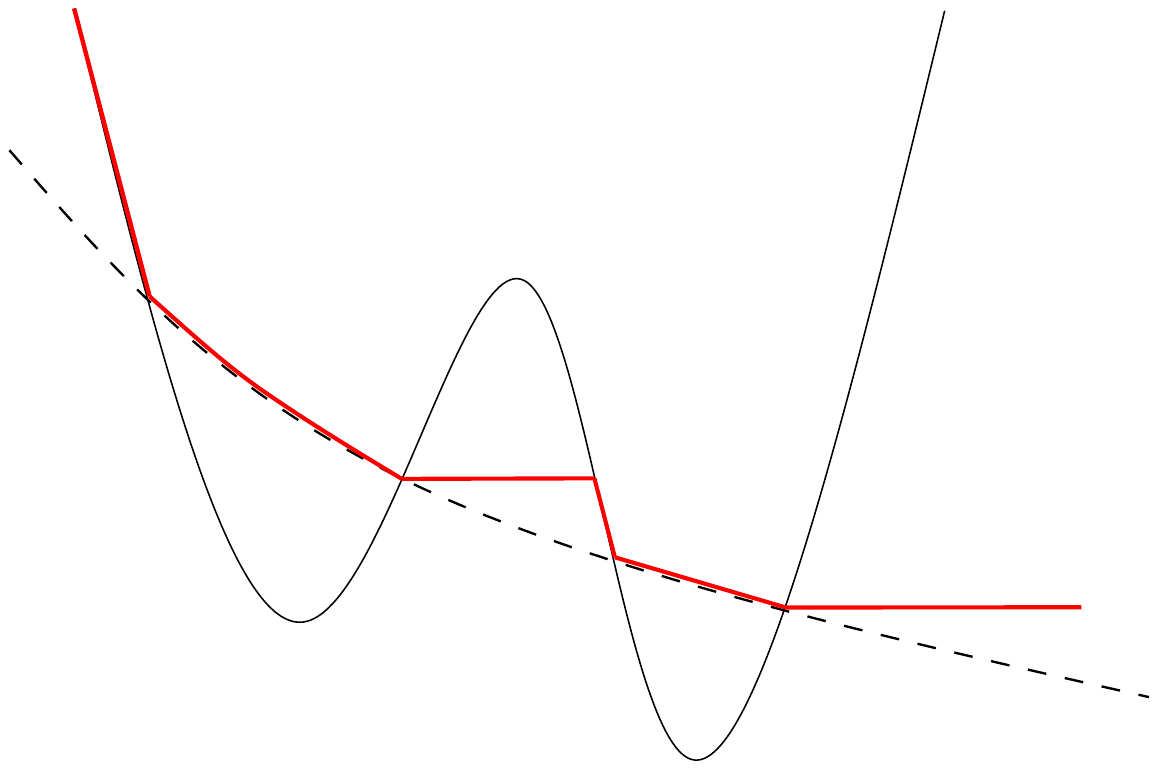}}%
    \put(0.69235535,1.00319973){\color[rgb]{0.16862745,0,0}\makebox(0,0)[lt]{\lineheight{1.25}\smash{\begin{tabular}[t]{l}$H$\end{tabular}}}}%
    \put(0.17436155,0.9153826){\color[rgb]{0.16862745,0,0}\makebox(0,0)[lt]{\lineheight{1.25}\smash{\begin{tabular}[t]{l}$F_0$\end{tabular}}}}%
    \put(0.22017919,0.33757135){\color[rgb]{1,0,0}\makebox(0,0)[lt]{\lineheight{1.25}\smash{\begin{tabular}[t]{l}$\underline R F_0$\end{tabular}}}}%
    \put(0,0){\includegraphics[width=\unitlength,page=2]{overlineRF02.pdf}}%
  \end{picture}%
\endgroup%

\def\svgwidth{.3\textwidth}
\begingroup%
  \makeatletter%
  \providecommand\color[2][]{%
    \errmessage{(Inkscape) Color is used for the text in Inkscape, but the package 'color.sty' is not loaded}%
    \renewcommand\color[2][]{}%
  }%
  \providecommand\transparent[1]{%
    \errmessage{(Inkscape) Transparency is used (non-zero) for the text in Inkscape, but the package 'transparent.sty' is not loaded}%
    \renewcommand\transparent[1]{}%
  }%
  \providecommand\rotatebox[2]{#2}%
  \newcommand*\fsize{\dimexpr\f@size pt\relax}%
  \newcommand*\lineheight[1]{\fontsize{\fsize}{#1\fsize}\selectfont}%
  \ifx\svgwidth\undefined%
    \setlength{\unitlength}{595.27559055bp}%
    \ifx\svgscale\undefined%
      \relax%
    \else%
      \setlength{\unitlength}{\unitlength * \real{\svgscale}}%
    \fi%
  \else%
    \setlength{\unitlength}{\svgwidth}%
  \fi%
  \global\let\svgwidth\undefined%
  \global\let\svgscale\undefined%
  \makeatother%
  \begin{picture}(1,1.41428571)%
    \lineheight{1}%
    \setlength\tabcolsep{0pt}%
    \put(0,0){\includegraphics[width=\unitlength,page=1]{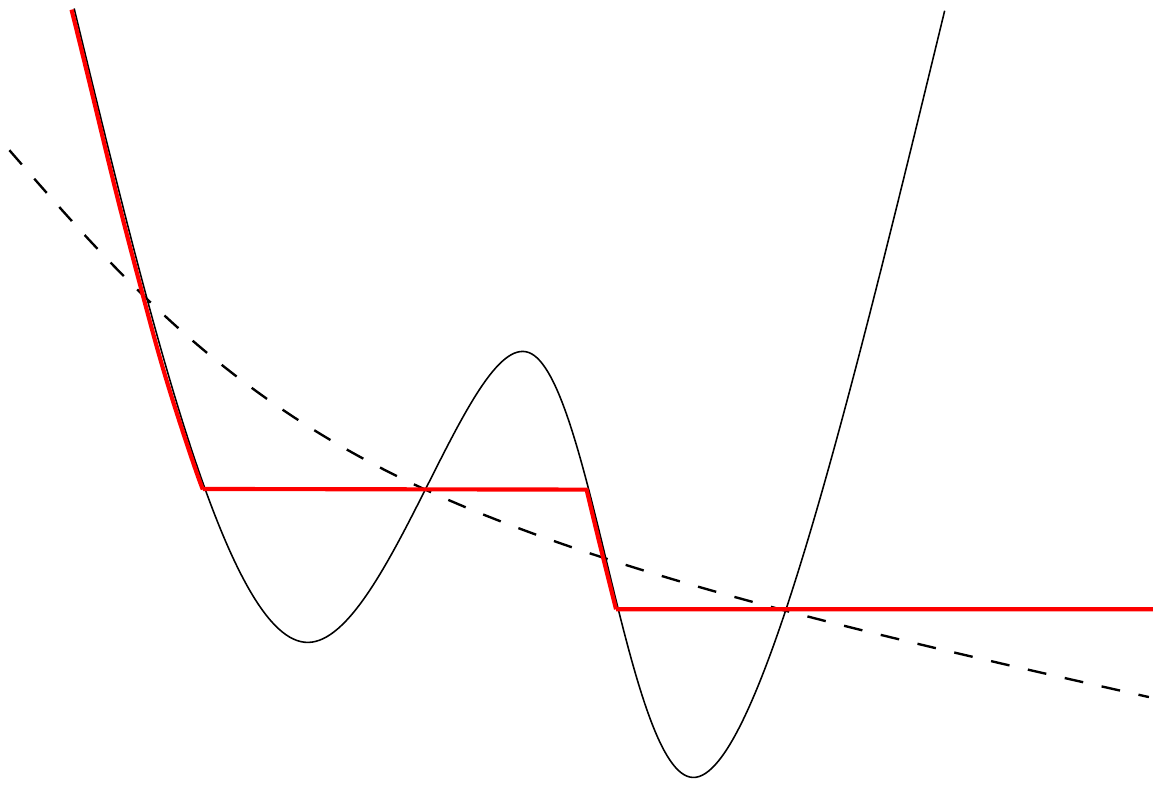}}%
    \put(0.66187824,0.98927249){\color[rgb]{0.16862745,0,0}\makebox(0,0)[lt]{\lineheight{1.25}\smash{\begin{tabular}[t]{l}$H$\end{tabular}}}}%
    \put(0.2335053,0.8800794){\color[rgb]{0.16862745,0,0}\makebox(0,0)[lt]{\lineheight{1.25}\smash{\begin{tabular}[t]{l}$F_0$\end{tabular}}}}%
    \put(0.32589947,0.41306879){\color[rgb]{1,0,0}\makebox(0,0)[lt]{\lineheight{1.25}\smash{\begin{tabular}[t]{l}$\mathfrak R F_0$\end{tabular}}}}%
    \put(0,0){\includegraphics[width=\unitlength,page=2]{RF0.pdf}}%
  \end{picture}%
\endgroup%

   \caption{Effects of $\underline R$, $\overline R$ and $\mathfrak R$ on $F_0$.
  The Hamiltonian $H$ is represented with a plain line, while a dashed line is used for the function $F_0$.
  The relaxation operators appear in red. We see that $\underline RF_0 \le F_0$ while $\overline RF_0 \ge F_0$.
  We can also observe that $\mathfrak R F_0 = \underline RF_0$ in $\{ F_0 \ge H \}$ and $\mathfrak R F_0 = \overline RF_0$ in $\{ F_0 \le H \}$.
  \label{F3}}
\end{figure}

\begin{ex}
In the totally degenerate case, \textit{i.e.} in the case where $F_0$ is constant,  the relaxed boundary function $\mathfrak R F_0$ is given by,
\[\mathfrak R F_0=\max(A,H_-)\quad \text{when}\quad F_0\equiv  const =A\quad \text{with}\quad H_-(p):=\inf_{(-\infty,p]} H.\]
This computation is used in the derivation of the relaxed Dirichlet condition, see the proof of Theorem~\ref{t:dirichlet}.
\end{ex}

The first main result of this work states that Guerand's relaxation operator
coincides with the one defined by \eqref{eq::r12}.
\begin{theo}[Guerand's operator and the relaxation operator coincide]\label{th::GR}
Assume  $H,F_0 \colon \R \to \R$ are continuous, $H$ is coercive and  $F_0$ is non-increasing and semi-coercive (in the sense of \eqref{eq:semi-coercive}). Then we have
\(\mathfrak R F_0=\mathfrak J F_0.\)
\end{theo}
\begin{rem}
The definition of Guerand's operator $\mathfrak J$ is recalled in Section~\ref{s4}, see Definition~\ref{defi:guerand-relaxed}.
\end{rem}

\subsubsection{The heterogeneous multidimensional setting}

If dimension is larger than or equal to $2$, then the relaxation operator can be defined by freezing tangential variables.
More precisely, if $x \in \partial \Omega$ and $n$ denotes the outward unit normal, then
$p \in \R^d$ is split into
\( p = p' - r n\)
for $p' \perp n$ and $r \in \R$. Then
\[ \bar H (r) = H (t,x,p'- r n) \quad \text{ and } \quad \bar F_0 (r) = F_0 (t,x,p' - r n).\]
We then define
\( \mathfrak{R} F_0 (t,x,p',r) = \mathfrak{R} \bar{F}_0 (r)\)
where the relaxation operator in the right hand side is computed with respect to the coercive Hamiltonian $\bar H$ and defined in \eqref{eq::r12}. 

We remark that the multi-dimensional relaxation operators can be written as,
\begin{equation}
  \label{defi:relax-op}
  \begin{cases}
  \underline{R} F_0 (t,x,p) &= \sup_{\rho \ge 0} \ (F_0 \wedge H)(t,x,p - \rho n ), \\
  \overline{R} F_0 (t,x,p) &= \inf_{\rho \le 0}\ (F_0 \vee H)(t,x,p- \rho n). 
\end{cases}
\end{equation}

\subsection{The Neumann and Dirichlet problems}
\label{subsec:neumann}

We now turn to the study of weak viscosity solutions of the Neumann problem. 
\begin{theo}[From Neumann to Godunov]\label{t:neumann}
  Any function $u \colon (0,T) \times \Omega \to \R$
  is a weak solution of Neumann problem \eqref{eq:HJ-neu}  if and only if  it is a strong solution  of
\[
\begin{cases}
u_t + H(t,x, Du)=0, & t \in (0,T),  x \in \Omega,\\
u_t +N(t,x,Du)=0, & t \in (0,T), x \in \partial \Omega
\end{cases}
\]
where $N$ is the classical Godunov flux associated to the Hamiltonian $\rho \mapsto H(t,x,p - \rho n)$,
\[
  N(t,x,p) =
  \begin{cases}
    \max \bigg\{ H(t,x,p- \rho n) : \rho \in [0,p \cdot n(x) + h(t,x)] \bigg \} & \text{ if } p \cdot n(x) + h(t,x) \ge 0, \\
    \min \bigg\{ H(t,x,p- \rho n) : \rho \in [p \cdot n (x) + h(t,x),0] \bigg\} & \text{ if } p \cdot n(x) + h(t,x) \le 0.
  \end{cases}
\]
\end{theo}
We remark that in dimension $1$ (taking $\Omega=(0,+\infty)$ to simplify), Theorem~\ref{t:neumann} can be expressed in terms of Godunov's flux.
Indeed, when $H$ and $h$ do not depend on $(t,x)$, we get $N(p) = G(h,p)$, where $G$ is the classical Godunov's flux defined later in \eqref{eq::r11}, and the weak Neumann boundary condition
is relaxed in $u_t + G(h,u_x)=0$. 
This formulation seems very natural; indeed, at the level of  the conservation law, it is expected that the spatial derivative $v:=u_x$  (at least formally) is an entropy solution of
$$\left\{\begin{array}{ll}
v_t + H(v)_x=0\quad , & \mbox{for}\quad x>0,\\
v(t,0)\in \mathcal G_{h} , & \mbox{for a.e.}\quad t\in (0,+\infty).
\end{array}\right.$$
where the set $\mathcal G_{h}$ is given by\footnotemark[1]
\footnotetext[1]{Notice that it is possible to show (similarly to the proof of Lemma \ref{lem::r3} below) that $u(t,x)=px+ \lambda t$ is a weak Neumann solution to (\ref{eq:HJ-neu}), if and only if $\lambda=-H(p)$ and
$$(p-h)\cdot \left\{G(h,p)-H(p)\right\}\ge 0$$
which is easily seen to be equivalent to $p\in \mathcal G_{h}$.
}

$$\mathcal G_{h}=\left\{p\in\R,\quad H(p)=G(h,p)\right\}.$$
It is easy to check that we have
$$\mathcal G_{h}=\left\{p\in \R,\quad \left\{\mbox{sign}(p-k)-\mbox{sign}(h-k)\right\}\cdot \left\{H(p)-H(k)\right\}\le 0 \quad \mbox{for all $k\in \R$}\right\}$$
which is nothing else that the well-known Bardos-Leroux-Nedelec (BLN) condition. This (BLN) condition that has been identified in \cite{BLN}, as the natural effective condition
associated to the desired Dirichlet condition for scalar conservation laws, in the vanishing viscosity limit.

In our weak/strong terminology, this shows in this example, that (BLN) condition is a strong boundary condition associated to the weak  Dirichlet boundary condition $v(t,0)=h$. Here the Dirichlet condition  can not always be satisfied strongly.
In other words, in this example, we see that relaxation of the boundary condition at the Hamilton-Jacobi level, selects the right choice of the effective boundary condition that is indeed satisfied strongly by a solution.

We refer the reader to Subsection \ref{ss.4}, for a further discussion on the relation between Hamilton-Jacobi equations with boundary conditions and scalar conservation laws with (Dirichlet type) boundary conditions.
\medskip

Notice that the Neumann problem has been adressed independently by P.-L. Lions and P. Souganidis \cite{zbMATH06821670} in the monodimensional case and the second author
with V. D. Nguyen \cite{MR3709301} in the case where the Hamiltonian is convex and the domain $\Omega$ is a half-space. In both works, the geometric setting corresponds to junctions
and the junction conditions of Kirchoff type can be handled. These conditions generalize the Neumann boundary condition to the junction setting. 
\bigskip

As far as the Dirichlet problem is concerned, the relaxed boundary
condition turns out to be an obstacle problem. 
\begin{theo}[Dirichlet to boundary obstacle problem]\label{t:dirichlet}
Consider a function $u \colon (0,T) \times \Omega \to \R$ which is
 weakly continuous at $(t,x)$ for all $t>0$ and $x \in \partial \Omega$, \textit{i.e.}
   \[u^*(t,x)=\limsup_{(s,y)\to (t,x),  y \in \Omega} u(s,y),\]
Then $u$ is a weak solution of Dirichlet problem \eqref{eq:HJ-dir}
if and only if  it is a strong solution  of
\[
\begin{cases}
u_t + H(t,x, Du)=0, & t \in (0,T),  x \in \Omega,\\
  \max \{ u-g, u_t + H_-(t,x,Du)\} =0,
                    & t \in (0,T), x \in \partial \Omega
\end{cases}
\]
where $n\colon \partial \Omega \to \R^d$ is the outward unit normal vector field and
\[
  H_- (t,x,p) = \inf_{\rho \le 0} H (t,x,p-\rho n(x)).
\]
\end{theo}

\subsection{Godunov's relaxation}

We show that relaxation is directly related to the classical Godunov's flux.
For the sake of simplicity, we present it in the monodimensional setting. 
We recall that this ``numerical'' flux is defined for $p,q \in \R$ by
\begin{equation}\label{eq::r11}
  G(q,p)=
  \begin{cases}
    \max_{[p,q]} H & \text{ if } p\le q,\\
    \min_{[q, p]} H & \text{ if }  p\ge q.
  \end{cases}
\end{equation}

\begin{theo}[Relaxation coincides with Godunov's relaxation]\label{th::r1}
  Assume $H,F_0\colon \R \to \R$ are continuous, $H$ is coercive and $F_0$ is non-increasing and semi-coercive.
  Then for any $p\in \R$, there is one and only one $\lambda \in \R$ such that  there exists $q\in \R$ with $\lambda=F_0(q)=G(q,p)$.
If $F_0G$ denotes  the map $p \mapsto \lambda$, then it coincides with the relaxation operator,
\[\mathfrak RF_0=F_0G.\]
\end{theo}
\begin{rem}
  For some technical reasons that will appear in the proof of this result, it makes more sense to define
  the action of Godunov's flux $G$ on the right of $F_0$ (rather than on the left).
\end{rem}

\subsection{Comments}

\paragraph{Self-relaxed boundary conditions.} We will see that the relaxed boundary condition cannot be further  relaxed, \textit{i.e.} it  satisfies $\mathfrak R(\mathfrak R F_0)=\mathfrak R F_0$. When a function $F_0$ satisfies $F_0=\mathfrak R F_0$, then we say that it is {\it self-relaxed}.

\paragraph{The lower non-increasing envelope of the Hamiltonian.}
The lower non-increasing envelope  $H_-$ of $H$ satisfies semi-coercivity condition (\ref{eq:semi-coercive}), it is self-relaxed, and for any boundary function $F_0$ satisfying (\ref{a:HandF0}), we have
\[
\mathfrak R F_0 \ge H_-.
\]
In other words, $H_-$ is the minimal self-relaxed boundary function. It corresponds to the natural condition that appears for state contraint problems with convex Hamiltonians, and can be seen as a sort of generalization of it to the case of non-convex and coercive Hamiltonian.
The previous inequality  implies that every continuous weak $F_0$-subsolution is indeed a strong $H_-$-subsolution (see Proposition \ref{pro::28}). 
This explains  (at least for a junction with a single branch) the observation made by P.-L. Lions and P. Souganidis \cite{lions2016viscosity} that only the supersolution condition has to be checked at the junction point. In other words, it is sufficient to check that the function is a weak (or strong) $H_-$-supersolution.

\paragraph{Weak continuity condition.}
Notice that when $F_0$ does not satisfy the semi-coercivity condition,
it is necessary to impose a weak continuity condition on the boundary,
\[ \forall (t_0,x_0) \in (0,+\infty) \times \partial \Omega,
  \quad u^*(t_0,x_0)=\limsup_{(s,y)\to (t_0,x_0), y \in \Omega} u(s,y).\]
to ensure that the conclusion of Theorem~\ref{t:main} holds true.
If none of these conditions is satisfied, then the conclusion
may be wrong, as shown in the counter-example~\ref{rem::32} below. It is due to J. Gerrand. 
Such a weak continuity condition appears for instance in the
work by G. Barles and B. Perthame \cite{zbMATH04131139} in which
they prove comparison principle for discontinuous viscosity solutions
of the Dirichlet problem (in the stationary case).
Such a condition also appears in \cite{zbMATH06713740} and in subsequent papers.

\paragraph{The stationary case.}
A version of Theorem~\ref{t:main} is still valid without changes in the definition of $\mathfrak R F_0$ for stationary equations like
$$\left\{\begin{array}{ll}
u+ H(x,Du)=0, & \text{for}\quad x\in \Omega,\\
u + F_0(x,Du)=0, & \text{for}\quad x\in \partial \Omega
\end{array}\right.$$
with adapted assumptions on $H,F_0$, and naturally adapted definitions of weak and strong viscosity solutions.

\subsection{Review of literature and known results}

\paragraph{Boundary conditions for viscosity solutions.}

The Dirichlet problem is considered in the first papers dealing with viscosity solutions, see \cite{zbMATH03736487,MR0690039,MR0732102}. We mentioned above that the weak continuity condition first appears in \cite{zbMATH04131139} where the authors prove a comparison principle for discontinuous viscosity solutions of HJ equations with Dirichlet boundary conditions. In this article, the boundary condition is imposed in the generalized sense recalled earlier.

The state-constraint condition is a boundary condition that has been identified early in the literature when Hamiltonians are convex.
H.~M.~Soner \cite{MR0838056} proved a general uniqueness result by constructing a special test function pushing contact points inside the domain.
As far as the Neumann boundary condition is concerned, it has been first adressed by P.-L.~Lions \cite{zbMATH03966826} for Hamiltonians that are not necessarily convex.

This first result for the Neumann boundary condition has been  generalized later
by G.~Barles \cite{zbMATH01288036}. In this work, he also constructed a test function à la Soner. 
The Neumann boundary condition is easily interpreted in the optimal control setting.

\paragraph{Convex Hamiltonians and optimal control.}

In 2007, A. Bressan and Y. Hong studied optimal control problems on stratified domains. The case of junctions is the simplest geometric setting of stratified domains. For such a geometry,
two groups of authors studied convex Hamilton-Jacobi equations: Y. Achdou, F. Camilli, A. Cutri and N. Tchou \cite{zbMATH06189399} on the one hand and the second and third authors together
with H. Zidani \cite{zbMATH06144570} on the other hand. At the same time, in the two domains setting,  G. Barles, A. Briani and E. Chasseigne \cite{zbMATH06198069,zbMATH06348168} developped
an intermediate approach mixing PDE and optimal control tools for convex Hamiltonians. 

In the monodimensional setting, solutions of a HJ equation are naturally associated with solutions of the corresponding scalar conservation law. In the two domains setting,
B. Andreianov, K. H. Karlsen, N. H. Risebro \cite{zbMATH06101925} developped a theory for existence and uniqueness from which the second and third authors took inspiration
to write \cite{zbMATH06713740}. We also mention the work by B. Andreianov and K. Sbihi \cite{zbMATH06429108} for the one domain problem in great generality.

Later, the second and third authors \cite{zbMATH06713740} introduced the notion of flux-limited solutions and cook up a PDE method generalizing
the method of doubling of variables to prove comparison principles. 
The case of networks is treated in \cite{zbMATH06713740} while \cite{MR3690310} is concerned with multi-dimensional junctions.
They observed that the  state constraint boundary conditions can be interpreted in terms of flux limiters, see \cite[Proposition~2.15]{zbMATH06713740}. J. Guerand treated
 the multidimensional case of state constraints in \cite{MR3695961}. 
We also mention that the second author together
with V. D. Nguyen \cite{MR3709301} addressed the case of parabolic equations degenerating to Hamilton-Jacobi equations at the (multi-dimensional) junction.

In  \cite{zbMATH06350884}, Z. Rao, A. Siconolfi and H. Zidani adopted a pure optimal control approach to deal with accumulation of components.
More recently, A. Siconolfi \cite{zbMATH07541881} proposed another PDE method based on the notion of maximal subsolutions under trace constraints to prove
a comparison principle on networks without loops.  Hamiltonians are convex and depend on the space variable and the uniqueness result holds true for uniformly continuous sub/supersolutions.

Motivated by the study of a homogeneization problem, the notion of flux-limited solutions has also been extended by Y. Achdou and C. Le Bris \cite{achdou:hal-03870614} for a convex HJ problem in $\R^d\backslash \left\{0\right\}$ supplemented with a condition at the origin.

These works have been extended mainly for optimal control problems on stratified domains by Barles, Chasseigne \cite{zbMATH06498621} (see also the recent work of Jerhaoui, Zidani \cite{zbMATH07655162}), and recently by the same authors in a book Barles, Chasseigne \cite{barles:hal-01962955} which is a reference book on the topic, including boundary conditions, junction problems in any dimensions, stratified problems,
in particular in relation with optimal control problems and convex Hamiltonians.

\paragraph{Non-convex Hamiltonians.}

J. Guerand \cite{zbMATH06731793} proved  comparison principles for non-convex HJ equations of evolution type posed in the half real line. She adressed both the coercive and non-coercive
cases. In order to prove such uniqueness results, she introduced a relaxation operator $\mathfrak{J}$ and  proved the equivalence between weak and strong solutions.

P.-L. Lions and P. Souganidis \cite{lions2016viscosity, zbMATH06821670} also studied Hamilton-Jacobi equations posed on junctions
in the non-convex case. In particular, they introduced a blow-up method to prove the comparison principle  between bounded
uniformly continuous sub- and super-solutions.

\subsection{Organisation of the article}

In Section \ref{s2}, we present the main properties of the relaxation operator $\mathfrak R$ and introduce the notion of characteristic points.
In Section \ref{s3}, we discuss relations between weak  and strong (viscosity) solutions and propose a new proof of
Theorem~\ref{t:main}  (see Theorem \ref{th:weak-strong}). In this section, 
 we also discuss existence and stability of weak viscosity solutions.
In Section \ref{s4}, we recall Guerand's relaxation formula, and show that it is equivalent to the new relaxation formula (Theorem \ref{th::GR}).
In Section \ref{s5}, we introduce Godunov's relaxation formula, and show that it is equivalent to the new relaxation formula (Theorem \ref{th::r1}).
In Section \ref{sec:neumann-dir}, we treat the case of Neumann and Dirichlet boundary conditions and prove Theorems~\ref{t:neumann} and \ref{t:dirichlet}.
We also discuss the link between the relaxation operator for HJ equations and  scalar conservation laws.


\paragraph{Notation.}

For $a,b \in \R$, $a \wedge b = \min (a,b)$ and $a \vee b = \max (a,b)$.

\section{Relaxation operators and characteristic points}\label{s2}

We recall that we always assume that $H,F_0$ satisfy \eqref{a:HandF0}.
In this section, we discuss properties of the relaxation operators.
For clarity, time, space and tangential variables are omitted throughout this section. 

\subsection{Relaxation operators}
We begin by some properties on the sub and super-relaxation operators. We recall that there are defined respectively in \eqref{eq::r16} and \eqref{eq::r17} and we refer to Figure \ref{F3} for a representation of the action of these operators on the function $F_0$.

\begin{lem}[First properties of the operators $\underline R$ and $\overline R$]
  \label{lem:relax-prop-1}
 Assume (\ref{a:HandF0}).
Then the functions  $\underline R F_0$ and $\overline R F_0$ are continuous, nonincreasing and semi-coercice, and
\[F_0 \wedge H \le \underline R F_0 \le F_0 \le \overline R F_0 \le F_0 \vee H \]
and \(\underline R  (\underline R  F_0)=\underline R  F_0\) and \(\overline R  (\overline R  F_0)=\overline R  F_0\).
\end{lem}
\begin{rem} \label{rem:relax-prop-1}
  We will use repeatedly the following easy consequences of this lemma:
  $\{F_0 \le H \} \subset \{ \underline R F_0 = F_0\}$ and $\{F_0 \ge H \} \subset \{ \overline R F_0 = F_0\}$. 
\end{rem}
\begin{proof}
  We only justify the properties satisfied by $\underline R F_0$ since proofs for $\overline R F_0$ are similar.
  We have by definition
\[
  (F_0 \wedge H)(p)\le \underline R F_0(p)=\sup_{q\ge p} (F_0 \wedge H)(q) \le \sup_{q\ge p} F_0(q)=F_0(p)
\]
where we have used the monotonicity of $F_0$. Moreover, by construction, $\underline R F_0$ is nonincreasing and  continuous.
The fact that $H$ is coercive and $F_0$ is semi-coercive implies that $F_0 \wedge H$ is also semi-coercive, and then $\underline RF_0$ is semi-coercive.

We set \(F:=\underline RF_0.\)
On the one hand, by coercivity of $H$, there exists some minimal $q^*\ge p$ such that
\[F(p)= \underline R F_0(p)=(F_0 \wedge H)(q^*).\]
Since $\underline R F$ is non-increasing, $q^* \ge p$ and $F(q^*)\ge (F_0\wedge H)(q^*)$, we have 
\[F(p)\ge \underline RF(p)\ge \underline RF (q^*) \ge  (F \wedge H)(q^*)\ge (F_0 \wedge H)(q^*)=F(p).\]
 The previous inequalities imply in particular that $F = \underline RF$. 
\end{proof}

We now define the \emph{relaxation operator}
\begin{equation}\label{eq:defi-relax}
(\mathfrak R F_0)(p):=\begin{cases}
\underline R F_0(p)&\quad \text{if}\quad F_0(p)\ge H(p)\\
\overline R F_0(p)&\quad \text{if}\quad F_0(p)\le H(p).
\end{cases}
\end{equation}
In particular, it satisfies \(|\mathfrak R F_0-H|\le |F_0-H|\)
as we will show next. In this sense, we see that $\mathfrak R F_0$ is closer to $H$ than $F_0$ itself.
\begin{lem}[Nice properties of the operator $\mathfrak R$]
  \label{lem:relax-op}
Assume \eqref{a:HandF0}.
The function $F:=\mathfrak  R F_0$ is well-defined, continuous, non-increasing, semi-coercive and satisfies
\begin{eqnarray}
\nonumber 
  \begin{cases}
    \underline R F=F=\overline R F\\
    \mathfrak R F=F
  \end{cases}
\\[1ex]
\nonumber F= \overline R(\underline R F_0)=\underline R(\overline R F_0) \\[1ex]
\label{eq::1}
\begin{cases}
F_0\le H &\quad \Longrightarrow \quad F_0\le \mathfrak R F_0 \le H,\\
F_0\ge H &\quad \Longrightarrow \quad F_0\ge \mathfrak R F_0 \ge H.
\end{cases}
\end{eqnarray}
For $H_-$ given by \(H_-(p)=\inf_{q\le p} H(q)\), the function $F_1:=F_0 \vee H_-$ is semi-coercive and satisfies
\[
\mathfrak R F_0= \mathfrak R F_1.
\]
\end{lem}
\begin{proof}
The proof is split in several steps. 
  
\paragraph{Step 1: preliminaries.}
We first notice that from Lemma~\ref{lem:relax-prop-1}, we have
\begin{equation}\label{eq:F0H}
\left\{\begin{array}{lll}
F_0(p)\le H(p) & \quad \Longrightarrow \quad \underline R F_0=F_0 \le \overline R F_0 \le H &\quad \text{at}\quad p\\
F_0(p)\ge H(p) & \quad \Longrightarrow \quad \overline R F_0=F_0 \ge \underline R F_0 \ge H &\quad \text{at}\quad p.
\end{array}\right.
\end{equation}
This implies that
$$F_0(p)= H(p) \quad \Longrightarrow \quad \overline R F_0=F_0 = \underline R F_0 =H\quad \text{at}\quad p.$$
Hence  the definition of $\mathfrak R F_0$ is equivalent to the following one,
\[
  (\mathfrak R F_0)(p):=
  \begin{cases}
    \underline R F_0(p)&\quad \text{if}\quad F_0(p)> H(p),\\
    F_0(p) &\quad \text{if}\quad F_0(p)= H(p),\\
    \overline R F_0(p)&\quad \text{if}\quad F_0(p) < H(p).
  \end{cases}
\]
In particular we see that $F:=\mathfrak R F_0$ is continuous, non-increasing and semi-coercive. 

\paragraph{Step 2: Effect of the operators on $F=\mathfrak R F_0$.}
We have
\[\underline R F_0 \le F=\mathfrak R F_0 \le \overline R F_0.\]
Hence, thanks to Lemma~\ref{lem:relax-prop-1} and the previous step,
\begin{align*}
  F &\ge \underline R F=\mathfrak R F \ge  \underline R ( \underline R F_0)= \underline R F_0 =F \quad \text{in}\quad \left\{F\ge H\right\} \\
  F &\le \overline R F=\mathfrak R F \le  \overline R ( \overline R F_0)= \overline R F_0 =F \quad \text{in}\quad \left\{F\le H\right\}.
\end{align*}
This implies that \(\mathfrak R F=F\). Moreover \eqref{eq:F0H} implies that $\{ F \le H \} \subset \{\underline R F =F\}$ and
$\{ F \ge H \} \subset \{ \overline R F =F \}$. We thus also get \(\underline R F= F =\overline R F.\)

\paragraph{Step 3: $\overline R (\underline R F_0)=F$ and $\underline R (\overline R F_0)=F$.}
We only prove the first equality since the proof of the second one is very similar. It amounts to prove that
\[\overline R (\underline R F_0)=
  \begin{cases}
	\underline R F_0 &\quad \text{in}\quad \{F_0>  H\}\\
	F_0 & \quad \text{in} \quad \{F_0=H\} \\
        \overline R F_0 &\quad \text{in}\quad \{F_0< H\}.
  \end{cases}
\]
The equality in the set $\{ F_0=H\}$ follows directly from Lemma~\ref{lem:relax-prop-1}. 

To check this equality in $\{F_0 > H\}$, we recall that Lemma~\ref{lem:relax-prop-1} implies that
$\{ F_0 > H \} \subset \{\underline R F_0 \ge H\}$, and then by Remark \ref{rem:relax-prop-1}, we get
$$\overline R(\underline R F_0)=\underline R F_0 \quad \text{on}\quad \left\{F_0>H\right\}.$$

To check this equality in $\{F_0<H\}$, we consider some maximal interval $(a,b)\subset \{F_0< H\}$.
Assume first that $a> - \infty$. 
In this case, we have $F_0(a)=H(a)$. Recalling that $\{ F_0 \le H \} \subset \{\underline R F_0 = F_0\}$ (see Lemma~\ref{lem:relax-prop-1}), we get that for $p\in (a,b)$,
\begin{align*}
\overline R (\underline R F_0)(p)&=\ \inf_{q\le p} (\underline R F_0 \vee H)(q)\\
&=\ \min\left\{\inf_{q\in [a,p]} (\underline R F_0 \vee H)(q),H(a)\right\}\\
&=\ \min\left\{\inf_{q\in [a,p]} (F_0 \vee H)(q),H(a)\right\}\\
&=\ \inf_{q\le p} (F_0 \vee H)(q)\\
&= \overline R F_0(p).
\end{align*}
Assume now that $a=-\infty$.
Then the same computation works with $a=-\infty$, $F_0(a)=H(a)=+\infty$, and $[a,p]$ replaced by $(-\infty,p]$.

\paragraph{Step 4: Proof of \eqref{eq::1}.} Combining \eqref{eq:F0H} and the fact that $\mathfrak R F=\overline R(\underline RF_0)=\underline R(\overline RF_0)$, we get the desired result.

\paragraph{Step 5: properties  of $F_1$.}
We have 
\[\overline R F_0(p)=\inf_{q\le p} (F_0 \vee H)(q)\]
and since $H_- \le H$, the function $F_1=F_0 \vee H_-$ satisfies 
$$\overline R F_1(p)=\inf_{q\le p} ((F_0 \vee H_-) \vee H)(q)=\inf_{q\le p} (F_0 \vee H)(q)=\overline R F_0(p).$$
Hence
\(\mathfrak R F_1=\underline R (\overline R F_1)= \underline R (\overline R F_0)=\mathfrak R F_0.\)
 Finally $F_1$ inherits semi-coercivity from  $H_-$.
\end{proof}

We now have the following tools.
\begin{lem}[Optimality and local properties of $\underline R F_0$]
  \label{lem:opti-loc}
Let $p\in \R$.
\begin{enumerate}[label=(\roman*)]
\item {\sc  (Optimality properties)} \label{l:loc}
  Let $q \ge p$ be minimal  such that
\[\underline R F_0(p)=(F_0 \wedge H)(q).\]
If $F_0(p)\ge  H(p)$ then
\[
  \begin{cases} 
    F_0(q)\ge H(q)\\
    \underline R F_0 = H(q) \quad \text{in}\quad [p,q]\\
    H < H(q) \quad \text{in}\quad [p,q).
  \end{cases}
\]
\item {\sc (Local properties)} \label{l:local}
If \(\underline R F_0(p)> H(p)\) then 
\(\underline R F_0\) is constant in  $[p-\eps,p+\eps)$ for some $\eps>0$. 
\end{enumerate}
\end{lem}
\begin{lem}[Optimality and local properties of $\overline R F_0$]
  \label{lem:opti-locb}
Let $p\in \R$.
\begin{enumerate}[label=(\roman*)]
\item {\sc (Optimality properties)} Let $q\le p$ be maximal such that \[\overline R F_0(p)=(F_0 \vee H)(q).\]
If
\(F_0(p)\le  H(p)\), then
\[\begin{cases}
F_0(q)\le H(q) \\
\overline R F_0= H(q) \quad \text{in}\quad [q,p]\\
H> H(q) \quad \text{in}\quad (q,p].
\end{cases}
\]
\item {\sc (Local properties)} \label{l:localb}
If \(\overline R F_0(p)< H(p)\)
then \(\overline R F_0\) is constant in  $(p-\eps,p+\eps]$ for some $\eps >0$.
\end{enumerate}
\end{lem}
As an immediate consequence of Lemmas \ref{lem:opti-loc} and \ref{lem:opti-locb} (using moreover definition (\ref{eq:defi-relax})), we get
\begin{cor}[Local properties of $\mathfrak R F_0$]
  \label{cor::14}
If \(\mathfrak R F_0(p)\not= H(p),\)
then \(\mathfrak R F_0\) is  constant in a neighbourhood of $p$. 
\end{cor}
We only do the proof of Lemma~\ref{lem:opti-loc} since the proof of Lemma~\ref{lem:opti-locb} is very similar. 
\begin{proof}[Proof of Lemma~\ref{lem:opti-loc}]
The proof is split in two steps. 
  
\noindent {\sc Optimality properties.}
We assume that \(F_0(p)\ge H(p)\) and $q\ge p$ is minimal such that 
\[\underline  R F_0(p)=(F_0 \wedge H)(q).\]

Using the coercivity of $H$ and the monotonicity of $F_0$, let us define $q_0 \in [p,+\infty)$ such that
\[q_0:=\sup \left\{q'\ge p,\quad F_0\ge  H  \text{ in } [p,q']\right\}.\]
It satisfies $F_0(q_0)=H(q_0)= \underline R F_0 (q_0)$ (see Lemma~\ref{lem:relax-prop-1}) and $q_0\ge p$.

We observe first that $q \in [p,q_0]$. Indeed, 
\begin{align*}
  \underline R F_0 (p) & = \max \left( \max_{q' \in [p,q_0]} H (q'), \underline R F_0 (q_0)  \right) \\
& = \max \left( \max_{q' \in [p,q_0]} H (q'), H (q_0)  \right) \\
& = \max_{q' \in [p,q_0]} H (q').
\end{align*}
We thus conclude that the maximum is reached for $q' \in [p,q_0]$ and since $q$ is minimal, we get $q \in [p,q_0]$. 

The fact that $q \le q_0$ implies that $F_0(q) \ge H(q)$.

Since $H(q) = \max_{[p,q_0]} H$, we also get from the minimality of $q$ that $H < H(q)$ in $[p,q)$. 

To finish with, monotonicity of $\underline RF_0$ implies that for any $q' \in [p,q)$, 
\[ \underline R F_0 (q) \le \underline R F_0 (q') \le \underline R F_0 (p) = (F_0 \wedge H) (q) = H(q) \le \underline R F_0 (q).\]
This series of inequalities yields that $\underline R F_0$ is constant, equal to $H(q)$.
\bigskip

\noindent {\sc Local properties.}
Keeping in mind that $(F_0 \wedge H) \le \underline R F_0 \le F_0$, if
\(\underline R F_0(p)> H(p)\)
then
\( F_0(p)\ge \underline R F_0(p)=H(q)> H(p),\) with $q$ defined above.
This implies that $q>p$  and so $\underline RF_0$ is constant in $[p,q]$. Using the monotonicity of $F_0$ and the continuity of $H$, we get also that there exists $\eps>0$ such that
$$H<H(q)\le F_0\quad {\rm on}\; [p-\eps, p].$$
Using the monotonicity of $\underline RF_0$, this implies, for all $p'\in [p-\eps,p]$, that
$$
\underline R F_0(q)\le \underline R F_0(p')=\max(\sup_{q'\in [p',p]}(F_0\wedge H)(q'), \underline R F_0(p))\le \max (H(q), \underline RF_0(q))=\underline RF_0(q).
$$
Hence $\underline R F_0$ is constant in $[p-\eps,q]$, with $q>p$. This yields the desired local property.
\end{proof}

\begin{lem}[Commutation of max/min with $\mathfrak R$]\label{lem::r90}
Assume that $H$ is continuous and coercive. If $F_a,F_b$ are continuous non-increasing, then
$$\mathfrak R(F_a\wedge F_b)=(\mathfrak R F_a)\wedge (\mathfrak R F_b)\quad \text{and}\quad 
\mathfrak R(F_a\vee F_b)=(\mathfrak R F_a) \vee (\mathfrak R F_b).$$
\end{lem}
\begin{proof}
We only prove $\mathfrak R(F_a\wedge F_b)=(\mathfrak R F_a)\wedge (\mathfrak R F_b)$ (the proof of the other relation with the max is similar).

\noindent \textsc{Step 1: commutation of $\min$ with $\overline R$.}
We have
$$\overline R (F_a\wedge F_b)(p)=\inf_{(-\infty,p]} (F_a\wedge F_b)\vee H=\inf_{(-\infty,p]} (F_a\vee H)\wedge (F_b\vee H)=\left\{\inf_{(-\infty,p]} (F_a\vee H)\right\}\wedge \left\{\inf_{(-\infty,p]} (F_b\vee H)\right\}$$
i.e.
$$\overline R (F_a\wedge F_b)=(\overline R F_a)\wedge (\overline R F_b).$$

\noindent \textsc{Step 2: commutation of $\min$ with $\underline R$.}
We first notice that
$$\underline R(F_a\wedge F_b)\le \underline R F_a,\underline R F_b  \quad \text{i.e.}\quad\underline R(F_a\wedge F_b)\le (\underline R F_a)\wedge (\underline R F_b).$$
Now we want to prove the reverse inequality. For $c=a,b$, let $q_c^{*} \ge p$ be minimal  such that $\underline R F_c(p)=(F_c\wedge H)(q^*_c)$.
Setting
$$q^*:=q^*_a\wedge q^*_b,$$
we get using the monotonicities of $F_a,F_b$
$$H(q^*)\ge H(q^*_a)\wedge H(q^*_b),\quad F_a(q^*)\ge F_a(q^*_a),\quad F_b(q^*)\ge F_b(q^*_b).$$ 
Hence
$$\underline R (F_a\wedge F_b)(p)=\sup_{[p,+\infty)} F_a\wedge F_b\wedge H\ge (F_a\wedge F_b\wedge H)(q^*)\ge (F_a\wedge H)(q^*_a)\wedge (F_b\wedge H)(q^*_b)=(\underline R F_a)\wedge (\underline R F_b)(p)$$
which is the reverse inequality. Hence we conclude that
$$\underline R(F_a\wedge F_b)= (\underline R F_a)\wedge (\underline R F_b).$$

\noindent \textsc{Step 3: conclusion.}
From Steps 1 and 2, we deduce that $\mathfrak R = \underline R \overline R$ also satisfies the same equality.
\end{proof}

\subsection{Characteristic points}

The following definition is concerned by the characteristic points. These characteristic points will be usefull in particular to reduce the set of test function in the definition of viscosity solutions (see Subsection \ref{subsec:reduction})
\begin{defi}[Characteristic points]
  \label{defi:charac}
\begin{enumerate}[label=(\roman*)]
\item $p$ is a \emph{positive characteristic point} of $F_0$ if
  $H(p) = F_0 (p)$ and 
  $H>H(p)$ in $(p,p+\eps)$ for some $\eps >0$.
  The set of positive characteristic points is denoted by   $\chi^+(F_0)$.
\item $p$ is a \emph{negative characteristic point} of $F_0$ if   $H(p) = F_0 (p)$ and 
  $H< H(p)$ in $(p-\eps,p)$ for some $\eps>0$. 
The set of negative characteristic points is denoted by   $\chi^-(F_0)$.
\item The set of all \emph{characteristic points} is denoted by $\chi (F_0)$,
  \textit{i.e.}  \(\chi (F_0):=\chi^+ (F_0)\cup \chi^- (F_0)\).
\end{enumerate}
\end{defi}
We present some example of characteristic points in Figure \ref{F4}. We would like to point out that in the case $d)$, the intersection point is not a characteristic point for $F_0$. Nevertheless, we will use this notion of characteristic point with the relaxation of $F_0$. In that case the left point of the plateau is in $\chi^-(\mathfrak R F_0)$ and the right point is in $\chi^+(\mathfrak R F_0)$.
\begin{figure}[!ht]
\centering\epsfig{figure=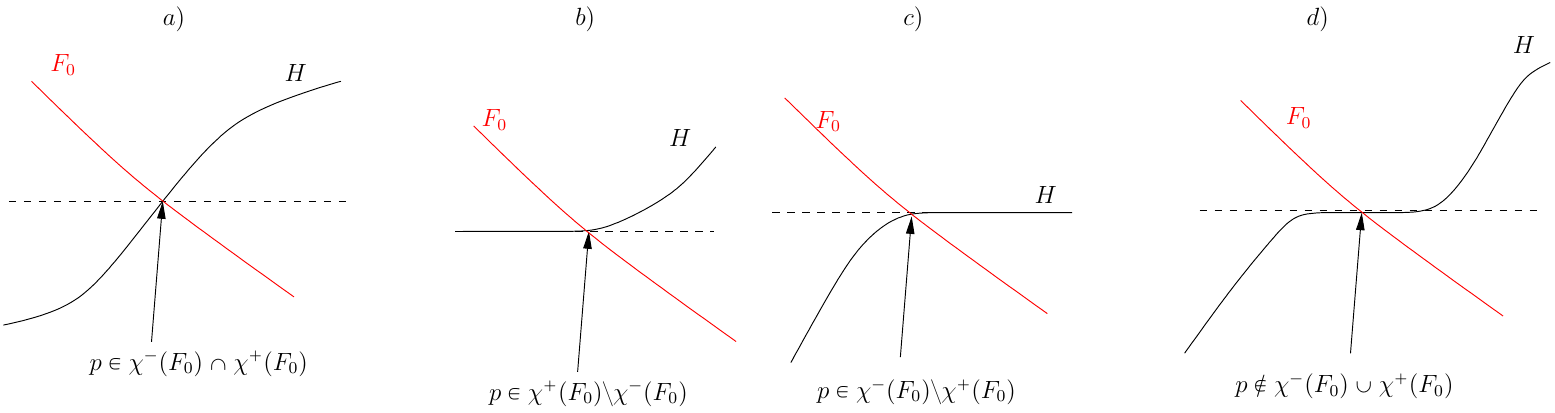,width=150mm}
\caption{Characteristic points of $F_0$ (along $H$)\label{F4}}
\end{figure}

In order to manipulate simply characteristic points, we use the notation introduced by J. Guerand in \cite{zbMATH06731793} and consider upper and lower points $p^\pm$ which only depend on $p$ and $H$.
The definition of $p^\pm$ is only related to the Hamiltonian $H$ while characteristic points give information about
the intersection of the graphs of $H$ and of $F_0$. An illustration of these points is given in Figure \ref{F5}
\begin{defi}[Upper and lower points]
Let $p \in \R$. 
  \begin{enumerate}[label=(\roman*)]
  \item If there exists $p_n \to p$ such that $p_n > p$ and $H(p_n) \le H(p)$, then the \emph{upper point}
    $p^+$ is  equal to $p$. If not, 
\[ p^+= \sup \left\{q>p :  H>H(p) \text{ in } (p,q)\right\}.\]
\item If there exists $p_n \to p$ such that $p_n < p$ and $H(p_n) \ge H(p)$, then the \emph{lower point}
    $p^-$ is  equal to $p$. If not, 
\[ p^-:= \inf \left\{q<p :  H<H(p) \text{ in }  (q,p)\right\}.\]
\end{enumerate}
\end{defi}
\begin{rem}
The coercivity of $H$ implies that \(-\infty<p^-\le p\le p^+\le +\infty.\)
\end{rem}

\begin{figure}[!ht]
\centering\epsfig{figure=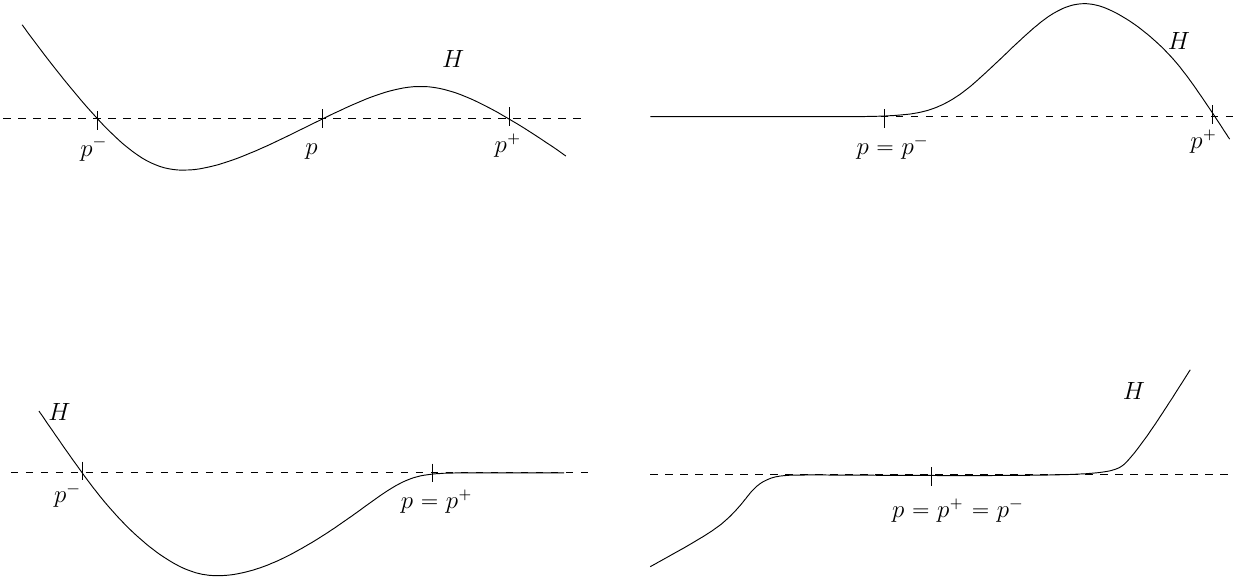,width=120mm}
\caption{Points $p^+$ and $p^-$ associated to $p$ (and $H$)\label{F5}}
\end{figure}

\begin{lem}[Characteristic points and relaxation operators]
  \label{lem:relax-prop}
Let $p \in \R$.
\begin{enumerate}[label=(\roman*)]  
\item If $p\in \chi^+(\overline R F_0)$, then  \(\overline R F_0\) is constant and equal to \(H(p)<H\) in \( (p,p^+)\). Moreover $H(p^+)=H(p)$ if $p^+<+\infty$.
\item If $p\in \chi^-(\underline R F_0)$, then \(\underline R F_0\) is constant and equal to \(H(p)=H(p^-)>H\) in \((p^-,p)\).
\end{enumerate}
\end{lem}
\begin{proof}
We only do the proof for negative characteristic points since the proof for positive ones is very similar.

Let $p \in \chi^-(\underline R F_0)$. Then $H(p) = \underline R F_0 (p)$, $p^-<p$ and $H < H(p)$ in $(p^-,p)$. For $p' \in (p^-,p)$, we then have
$H(p') < H (p ) =\underline R F_0(p) \le \underline R F_0 (p') \le F_0 (p')$. This implies
\begin{align*}
\underline R F_0 (p^-) = \max ( \sup_{[p^-,p]} H, \underline RF_0 (p)) = \underline R F_0 (p).
\end{align*}
Since $\underline RF_0$ is non-increasing, this yields the desired result. 
\end{proof}

\begin{cor}[Property of $\mathfrak R F_0$] \label{cor::11}
The function $\mathfrak R F_0$ satisfies
\[
  \mathfrak R F_0=\mathrm{constant} =H(p) \quad
  \begin{cases}
    \text{in } [p^-,p]& \quad \text{if}\quad p\in \chi^-(\mathfrak R F_0),\\
    \text{in } [p,p^+]\cap \R& \quad \text{if}\quad p\in \chi^+(\mathfrak R F_0).
  \end{cases}
\]
\end{cor}
\begin{rem}\label{rem::local}
In Corollary \ref{cor::11}, we only need $[p,p^+]\cap \R$ instead of $[p,p^+]$ in the special case where $p^+=+\infty$.
\end{rem}
\begin{proof}
We only do the proof for negative characteristic points since the proof for
positive ones is similar. 

Let $F_1 = \overline R F_0$. In particular $\mathfrak R F_0 = \underline R F_1$. 
If $p \in \chi^- (\mathfrak R F_0)=\chi^-(\underline R F_1)$, then Lemma~\ref{lem:relax-prop} implies that
for $p' \in (p^-,p)$, we have in particular $F_1 (p') \ge \underline R F_1 (p') > H(p')$.
This implies that  $(p^-,p) \subset \{ F_1 > H \}$. Moreover,
\[ \{ F_1 > H \} = \{ \overline R F_0 > H \} \subset \{ F_0 > H \} \]
since $\{F_0 \le H \} \subset \{\overline R F_0 \le H \}$ by Lemma~\ref{lem:relax-prop-1}. 
By definition of $\mathfrak R F_0$,
we have $\mathfrak R F_0 = \underline R F_0$ in $\{ F_0 > H \}$ and in particular in $(p^-,p)$.
We conclude by Lemma \ref{lem:relax-prop} that $\mathfrak R F_0$ is constant and equal to $H(p)$ in $(p^-,p)$. By continuity, we get the result in $[p^-,p]$.
\end{proof}

We also have another corollary of the previous results.
\begin{cor}[Values of $\mathfrak R F_0$ at its characteristic points] \label{cor::30}
We have $\mathfrak R F_0\le F_0$ in $\chi^- (\mathfrak R F_0)$
and $\mathfrak R F_0\ge F_0$ in $\chi^+ (\mathfrak R F_0)$. 
\end{cor}
\begin{proof}
  We only do the proof for negative characteristic points since the proof for positive ones is similar.
  
Let \(F=\mathfrak R F_0\) and \( p\in \chi^- (F). \)
This means 
\[ H< F(p)=H(p) \text{ in } (p^-,p)\not=\emptyset .\]
Since $F$ is non-increasing, this implies that 
\[ H < F = \mathfrak R F_0  \text{ in } (p^-,p). \]
In other words, $(p^-,p) \subset \{ F > H \}$. 
Lemma~\ref{lem:relax-op} implies that $\{ F > H \} \subset \{F_0 > H\}$. Hence $(p^-,p)\subset \{F_0>H\}$. By continuity of $F_0$ and $H$, we then get $F_0(p)\ge H(p)$ and by Lemma \ref{lem:relax-op} $F_0(p)\ge \mathfrak R F_0(p)$.
\end{proof}

\section{Viscosity solutions: properties, stability and existence}
\label{s3}

In this section, time, space and tangential variables are not omitted anymore.
We first discuss the notion of viscosity solutions and then explain how to reduce
the set of test functions for verifying that a function is indeed a strong viscosity solution.
As an application, we get our first main result, see Theorem~\ref{t:main} in the introduction
and Theorem~\ref{th:weak-strong} below. 

\subsection{Definitions of weak  and strong viscosity solutions}

We consider two notions of  viscosity solutions for the boundary value problem \eqref{eq:HJ-bc}. Weak viscosity solutions
are useful to get existence since they are naturally stable. Strong viscosity solutions are useful to prove uniqueness.

Before defining weak and strong viscosity solutions of \eqref{eq:HJ-bc}, we recall that
a function $\varphi$ touches a function $u$ from above (resp. from below) in a set $Q$ at a point $P_0 \in Q$ if
$\varphi\ge u$ in $Q$ (resp. $\varphi\le u$ in $Q$) and $u = \varphi$ at $P_0$. We also recall
that if a function $u$ is locally bounded from below (resp. from above), then its lower semi-continuous envelope $u_*$ (resp. upper semi-continuous envelope $u^*$)
is the largest lower semi-continuous function lying below $u$ (resp. smallest upper semi-continuous function lying above $u$).

In order to define weak and strong viscosity solutions of the three boundary value problems \eqref{eq:HJ-bc}, \eqref{eq:HJ-neu}
and \eqref{eq:HJ-dir}, we consider a real-valued continuous function $L =L(t,x,v, p_0,p)$ such that
\begin{equation}
  \label{a:L}
  L \colon (0,+\infty) \times \Omega \times \R \times \R \times \R^d \to \R \text{ is non-decreasing in $v$, $p_0$ and
 $p \cdot n(x)$.}
\end{equation}
The associated boundary value problem is the following one,
\begin{equation}\label{eq:HJ-gen}
\begin{cases}
u_t + H(t,x,Du)=0, &  t>0, x \in \Omega,\\
L(t,x,u,u_t,Du)=0, & t>0, x \in \partial \Omega.
\end{cases}
\end{equation}
The corresponding functions $L$ are respectively
$L = p_0 + F_0 (t,x, p)$, $L = v -g(t,x)$ and $L = p \cdot n(x) + h(t,x)$. 

\begin{defi}[Weak viscosity solutions]  \label{defi:weak-visc-sol}
Let $Q =(0,+\infty) \times \Omega$ and $u: Q \to \R$. 
\begin{enumerate}[label=(\roman*)]
\item
Let  $u$ be upper semi-continuous. We say that $u$ is a \emph{weak $L$-subsolution} of \eqref{eq:HJ-bc}
if for any point $P_0=(t_0,x_0)\in Q$, and any $C^1$ function $\varphi$ touching $u$ from above, 
then
 \begin{align*}
   \text{ if } x_0 \in \Omega, &  \quad \varphi_t + H(t,x,D \varphi) \le 0 & \text{ at } P_0\\
   \text{ if } x_0 \in \partial \Omega, &  \quad \text{ either } \quad \varphi_t + H(t,x,D \varphi) \le 0 \quad \text{ or } \quad L(t,x,\varphi,\varphi_t,D \varphi) \le 0,& \text{ at } P_0.
 \end{align*}
\item 
Let  $u$ be lower semi-continuous. We say that $u$ is a \emph{weak $L$-supersolution} of \eqref{eq:HJ-bc}
if for any point $P_0=(t_0,x_0)\in Q$, and any $C^1$ function $\varphi$ touching $u$ from below, then
\begin{align*}
  \text{ if } x_0\in \Omega, & \quad \varphi_t + H(t,x,D\varphi) \ge 0,& \text{ at } P_0\\
\text{ if } x_0 \in \partial \Omega, & \quad \text{either} \quad \varphi_t + H(t,x,D\varphi) \ge 0\quad \text{ or }\quad L(t,x,\varphi,\varphi_t,D\varphi) \ge 0,& \text{ at } P_0.
\end{align*}
\item 
  Let $u$ be locally bounded. We say that $u$ is a \emph{weak $L$-solution} (weak viscosity solution) of \eqref{eq:HJ-bc}, if $u^*$ is a weak $L$-subsolution  of \eqref{eq:HJ-bc},
  and $u_*$ is a weak $L$-supersolution of \eqref{eq:HJ-bc}. 
  \end{enumerate}
\end{defi}

\begin{defi}[Strong viscosity solutions]  \label{defi:strong-visc-sol}
Let $Q:= (0,+\infty)\times \Omega$ and $u: Q \to \R$. 
\begin{enumerate}[label=(\roman*)]
\item
Let  $u$ be upper semi-continuous. We say that $u$ is a \emph{strong $L$-subsolution} of \eqref{eq:HJ-bc}
if for any point $P_0=(t_0,x_0)\in Q$, and any $C^1$ function $\varphi$ touching $u$ from above, 
then
 \begin{align*}
   \text{ if } x_0 \in \Omega, &  \quad \varphi_t + H(t,x,D \varphi) \le 0 & \text{ at } P_0\\
   \text{ if } x_0 \in \partial \Omega, & \quad  L(t,x,\varphi,\varphi_t,D \varphi) \le 0,& \text{ at } P_0.
 \end{align*}
\item 
Let  $u$ be lower semi-continuous. We say that $u$ is a \emph{strong $L$-supersolution} of \eqref{eq:HJ-bc}
if for any point $P_0=(t_0,x_0)\in Q$, and any $C^1$ function $\varphi$ touching $u$ from below, then
\begin{align*}
  \text{ if } x_0>0, & \quad \varphi_t + H(t,x,D\varphi) \ge 0,& \text{ at } P_0\\
\text{ if } x_0 =0, & \quad L(t,x,\varphi,\varphi_t,D\varphi) \ge 0,& \text{ at } P_0.
\end{align*}
\item 
  Let $u$ be locally bounded. We say that $u$ is a \emph{strong $L$-solution} (strong viscosity solution) of \eqref{eq:HJ-bc}, if $u^*$ is a strong $F_0$-subsolution of \eqref{eq:HJ-bc},
  and $u_*$ is a strong $F_0$-supersolution of \eqref{eq:HJ-bc}. 
  \end{enumerate}
\end{defi}
\begin{rem}
  In the case where $L = u_t + F_0(t,x,Du)$, weak/strong $L$-sub/super-solutions are simply called weak/strong $F_0$-sub/super-solutions.
\end{rem}

\subsection{Reducing the set of test functions}

\subsubsection{Critical normal slopes and weak continuity}

We consider the equation without the boundary condition,
\begin{equation}\label{eq:HJ}
u_t+H(t,x,Du)=0 \quad \text{ in } Q
\end{equation}
where we recall that $Q$ denotes $(0,+\infty) \times \Omega$ and $\Omega$ is a $C^1$ domain of $\R^d$.
The regularity of the domain amounts to assume  that for all $x_0 \in \partial \Omega$,  there exists $r_0>0$ such that
\begin{equation}
\label{e:pomega}  \Omega \cap B_{r_0} (x_0) = \{ (x',x_d) \in B_{r_0} (x_0) :  x_d > \gamma (x') \}
\end{equation}
for
some $C^1$ function $\gamma \colon \R^{d-1} \to \R$ such that $\gamma (x_0')=0$ and $D' \gamma (x_0')=0$ where $D'$
denotes the derivative with respect to $x'$.  In particular, $n(x_0) = (0,-1) \in \R^{d-1} \times \R$.
The following lemma is proved in  \cite{MR3695961} for Hamiltonians that do not depend on $(t,x)$
and that have convex sub-level sets. The reader can check that neither the $(t,x)$ dependency nor the quasi-convex assumption
play a role in the proof. 
\begin{lem}[Critical normal slope for supersolutions -- \cite{MR3695961}]
  \label{lem:relax-op0b}
  Assume that $H$ is continuous and coercive and $\partial \Omega$ is $C^1$.
  Let $u: Q \to \R$ be lower semi-continuous. Assume that $u$ is a viscosity supersolution of \eqref{eq:HJ}
  and let $\varphi$ be a test function touching $u$ from below at $P_0:=(t_0,x_0)$ with $t_0>0$
  and $x_0 \in \partial \Omega$. Let $\gamma$ be a $C^1$ function and $r_0>0$ such that \eqref{e:pomega} holds true.
  Then the \emph{critical normal slope} defined by
  \[
    \overline p:= \sup\left\{ p\in \R,\quad \exists r \in (0,r_0),\quad \varphi(t,x)+ p (x_d - \gamma(x'))  \le u(t,x)
      \quad \text{for all}\quad (t,x)\in B_r (t_0,x_0) \cap Q \right\}
  \]
  is non-negative. 
If it is finite ($\overline p<+\infty$) then $\varphi_t + H(t,x,D \varphi-\overline p n(x_0))\ge 0$ at $P_0$.
\end{lem}
\begin{rem}
  In the case where $\Omega$ is a half space
  (\textit{i.e.} when $\partial \Omega$ is a hyperplane) and the Hamiltonian is quasi-convex,
  this lemma is proved in \cite[Lemma~3.4]{MR3709301}. 
\end{rem}

We now get a similar result for subsolutions. In this case, the critical normal slope is necessarily finite. 
\begin{lem}[Critical normal slope for subsolutions -- \cite{MR3695961}]
  \label{lem:relax-op0}
  Assume that $H$ is continuous and coercive and $\partial \Omega$ is $C^1$.
  Let $u: Q \to \R$ be upper semi-continuous. Assume that $u$ is a viscosity supersolution of \eqref{eq:HJ}
  and let $\varphi$ be a test function touching $u$ from below at $P_0:=(t_0,x_0)$ with $t_0>0$
  and $x_0 \in \partial \Omega$. Let $\gamma$ be a $C^1$ function and $r_0>0$ such that \eqref{e:pomega} holds true.
  Then the \emph{critical normal slope} defined by
  \[
    \overline p:= \inf\left\{ p\in \R,\quad \exists r \in (0,r_0),\quad \varphi(t,x)+ p (x_d - \gamma(x'))  \ge u(t,x)
      \quad \text{for all}\quad (t,x)\in B_r (t_0,x_0) \cap Q \right\}
  \]
  is non-positive.  If 
\begin{equation}\label{eq::21}
u^*(t_0,0)=\limsup_{(s,y)\to (t_0,0),\quad y>0} u(s,y)
\end{equation}
then it is finite ($\underline p>-\infty$) and $\varphi_t + H(t,x,D \varphi-\underline p n(x_0))\le 0$ at $P_0$.
\end{lem}
\begin{rem}
  In the case where $\Omega$ is a half space
  (\textit{i.e.} when $\partial \Omega$ is a hyperplane) and the Hamiltonian is quasi-convex,
  this lemma is proved in \cite[Lemma~3.4]{MR3709301}. 
\end{rem}
Notice that Condition \eqref{eq::21} is always satisfied for subsolutions of \eqref{eq:HJ-gen} when $H$ is coercive and $L$ is semi-coercive.
\begin{lem}[Weak continuity of weak subsolutions]
  \label{lem:weak-cont}
  Assume  that $H$ and $L$ are continuous, $H$ is coercive and
  $\lambda \mapsto L(t,x,v,p_0,p - \lambda n(x))$ is non-increasing and semi-coercive for all $(t,x,v,p_0,p)$,
  \[ \inf_{p' \perp n (x)} L(t,x,v,p_0,p' + \lambda n(x)) \to + \infty \text{ as } \lambda \to +\infty.\]
  If $u$ is a weak $L$-subsolution of \eqref{eq:HJ-gen}, then for all $t>0$, we have
\[u^*(t,x)=\limsup_{(s,y)\to (t,x),  y \in \Omega} u(s,y).\]
\end{lem}
\begin{proof}
  In the case where $\Omega$ is a half-space, the result corresponds to \cite[Lemma~2.3]{MR3709301}.
  The reader can check that the convexity of sub-level sets of $H$ are not used in this proof and
  that the only needed assumptions are the ones from the statement.

  In the case where $\Omega$ is a $C^1$ domain, we consider $x_0 \in \partial \Omega$ and
  $r>0$ and a $C^1$ function $\gamma \colon \R^{d-1} \to \R$ such that \eqref{e:pomega} holds true.
  We reduce to the case of the half-space by considering the function $\bar u (t,x)$ defined by $\bar u (t,x) = u (t,x',\gamma(x') + x_d)$.
  It is a weak $\bar L$-subsolution of \eqref{eq:HJ-gen} in an open ball centered $(t_0,x_0',0)$ intersected with $\{x_d > 0\}$ with $\bar H$ and $\bar L$ given
  \begin{align*}
    \bar L (t,x',x_d,v,p_0,p',p_d)&= L (t,x',x_d+\gamma (x'),v,p_0,p' - D'\gamma (x') p_d, p_d) \\
    \bar H (t,x,p) &= H (t,x',x_d+\gamma (x'), p' - D'\gamma (x') p_d, p_d).
  \end{align*}
  One can choose $r>0$ such that $|D'\gamma (x')| <1/2$ in $B_r (t_0,x_0)$. With such a choice at hand,
  we have $|p' - D'\gamma (x') p_d|+ |p_d| \ge |p'|+\frac12 |p_d|$ and this ensures the coercivity of
  $\bar H$. Moreover, the assumption on $L$ implies that $\bar L$ is semi-coercive.
  The weak continuity of $\bar u$ at $(t_0,x'_0,0)$ implies the weak continuity of $u$ at $(t_0,x'_0,x^0_d)$. 
\end{proof}

\subsubsection{Reduction of the set of test functions}\label{subsec:reduction}

In the two following results, we do not assume that $F_0$ is semi-coercive.
\begin{pro}[Reducing the set of test functions for strong subsolutions]
  \label{pro:reduced}
Assume that $H,F_0$ satisfy \eqref{a:HandF0}. 
Let $u: Q \to \R$ be upper semi-continuous and be a subsolution of \eqref{eq:HJ} in $Q \cap B_r(t_0,x_0)$ with $x_0 \in \partial \Omega$
with $r$ and $\gamma$ such that \eqref{e:pomega} holds true. We assume that
\[u^*(t_0,x_0)=\limsup_{(s,y)\to (t_0,x_0), y \in \Omega} u(s,y).\]

We then consider the class of test functions of the form
\begin{equation}\label{eq::23}
\varphi(t,x)=\psi(t,x')+ p_d x_d
\end{equation}
with $\psi$ continuously differentiable in $(t,x')$ and $p_d$ a negative characteristic point of $q_d \mapsto \underline R F_0(t_0,x_0,p'_0,q_d)$  where $p'_0=D'\psi(t_0,x'_0)$.

If for any $\varphi$ of the form \eqref{eq::23} touching $u$ from above at $P_0=(t_0,x_0)$, we have
\[
\varphi_t + \underline R F_0(t,x,D \varphi)\le 0\quad \text{at}\quad P_0
\]
then $u$ is a strong $\underline R F_0$-subsolution of \eqref{eq:HJ-bc} at $P_0$.
\end{pro}
\begin{proof}
  Let $\phi$ be an arbitrary test function touching $u$ from above at $P_0=(t_0,x_0)$ with $t_0>0$.
  Let $\lambda:=-\phi_t(P_0)$. We want to show that
\begin{equation}\label{eq::24}
\underline R F_0(t,x,D\phi)\le \lambda\quad \text{at}\quad P_0.
\end{equation}

Let  $\underline p\in (-\infty,0]$ be given by Lemma \ref{lem:relax-op0}. In particular,
\(H(t,x,D\phi - \underline p n(x)) \le \lambda\) at $P_0$.
Let $D \phi (P_0)=(p'_0,p^0_d)$ and $\underline p_d^0:=p_d^0+\underline p$.
Let us drop the $(t_0,x_0,p'_0)$ dependency for clarity.
We thus know that $H(\underline p_d^0) \le \lambda$. 

If $\underline R F_0(p_d^0)\le H(\underline p_d^0)$,
then we get \eqref{eq::24}. 
We are left with treating the case $\underline R F_0(p_d^0) > H(\underline p_d^0)$.
In this case, we have  $\underline R F_0(\underline p_d^0)\ge \underline R F_0(p_d^0)> H(\underline p_d^0)$
and Lemma \ref{lem:opti-loc} implies that
$\underline R F_0$ is constant in  $[\underline p_d^0,\underline p_d^0+\varepsilon)$ for some $\eps >0$.
From the coercivity of $H$ and the monotonicity of $F$, we also deduce that there exists some $p_* > \underline p_d^0$ such that
\[ \underline R F_0= \mathrm{const} > H\quad \text{on}\quad [\underline p_d^0,p_*) \]
with $\mathrm{const}= \underline R F_0(p_*)=H(p_*)$. In other words,  $p_*$ is a negative characteristic point of $\underline R F_0$: $p_*\in \chi^-(\underline R F_0)$. 

We now write $p_* = \underline p_d^0 + \delta = p_d^0 + (\underline p + \delta)$ for some $\delta >0$. Moreover,
the definition of $\underline p$ from Lemma~\ref{lem:relax-op0} implies that  there exists $r_0>0$ such that
we have
\[\phi(t,x) + (\underline p + \delta/2) (x_d-\gamma (x')) \ge u(t,x) \text{ in } B_{r_0}(t_0,x_0) \cap Q.\]
 Moreover,
\[\phi (t,x) \le \phi (t,x',\gamma (x')) + (p_d^0 + \delta/2) (x_d-\gamma(x')) \text{ in } B_{r_1} (t_0,x_0) \cap Q\]
 for some $r_1 < r_0$. 
Hence,
\[\varphi(t,x):=\phi(t,x',\gamma (x')) + \underset{p_*}{\underbrace{(p_d^0+\underline p + \delta)}} (x_d-\gamma(x')) \ge u(t,x) \text{ in } B_{r_1} (t_0,x_0) \cap Q.\]
By assumption, we have 
\(\lambda\ge \underline RF_0(p_*)= \underline R F_0 (\underline p_d^0) \ge \underline R F_0(p_d^0)\)
which in turn yields \eqref{eq::24}. 
\end{proof}
As far as strong supersolutions are concerned, it is not necessary to impose a weak continuity assumption, and we show similarly the following result.
\begin{pro}[Reducing the set of test functions for strong supersolutions]
  \label{pro:reducedb}
Assume that $H,F_0$ satisfy \eqref{a:HandF0}. 
Let $u: Q \to \R$ be lower semi-continuous and be a viscosity supersolution of \eqref{eq:HJ} in $Q \cap B_r(t_0,x_0)$ with $x_0 \in \partial \Omega$
with $r$ and $\gamma$ such that \eqref{e:pomega} holds true. 

We then consider the class of test functions of the form
\begin{equation}\label{eq::23b}
\varphi(t,x)=\psi(t,x')+ p_d x_d
\end{equation}
with $\psi$ continuously differentiable and $p_d$ a positive characteristic point of $q_d \mapsto \overline R F_0(t_0,x_0,p'_0,q_d)$  where $p'_0=D'\psi(t_0,x'_0)$.

If for any $\varphi$ of the form \eqref{eq::23b} touching $u$ from below at $P_0=(t_0,x_0)$, we have
\[
\varphi_t + \overline R F_0(t,x,D \varphi)\ge 0\quad \text{at}\quad P_0
\]
then $u$ is a strong $\overline R F_0$-supersolution of \eqref{eq:HJ-bc} at $P_0$.
\end{pro}

\subsection{Weak $F_0$-solutions are  strong $\mathfrak R F_0$-solutions}

\begin{lem}[Weak $F_0$ sub/super-solutions are strong $\underline R F_0/\overline RF_0$ sub/super-solutions]
  \label{lem:relax-op7}
\begin{enumerate}[label=(\roman*)]
\item Let $u: Q\to \R$ be upper semi-continuous.
Then  $u$ is a weak $F_0$-subsolution of \eqref{eq:HJ-bc} if and only if $u$  is a strong $\underline R F_0$-subsolution  of \eqref{eq:HJ-bc} .
\item \label{2romain} Let $u \colon Q \to \R$ be lower semi-continuous.
  Then $u$ is a weak $F_0$-supersolution  of \eqref{eq:HJ-bc}  if and only if $u$ is a strong $\overline R F_0$-supersolution  of \eqref{eq:HJ-bc}.
\end{enumerate}
\end{lem}
\begin{proof}
We only prove the result for subsolutions since the case of supersolutions is treated similarly. 
\bigskip

\noindent {\sc Weak implies strong.}
Assume that $u$ is a weak $F_0$-subsolution. Consider a test function $\phi$ touching $u$ from above at $P_0=(t_0,x_0)$ with $t_0>0$
and $x_0 \in \partial \Omega$. Let $r_0>0$ and $\gamma \in C^1(\R^{d-1})$ such that
\eqref{e:pomega} holds true. Then for any $\bar q\ge 0$, consider
\[ \varphi(t,x):=\phi(t,x)+ \bar q (x_d - \gamma (x'))\]
which is also touching $u$ from above at $P_0$. 
Then, either the equation or the boundary condition is satisfied at $P_0$, 
\[\varphi_t + (F_0 \wedge H)(t,x,D \varphi)\le 0\quad \text{at}\quad P_0.\]
We used the fact that $D' \gamma (x'_0)=0$. 
With $p:=D \phi(P_0)$, the previous inequality reads,
\[\phi_t(P_0) +(F_0 \wedge H)(p-\bar q n (x_0))\le 0.\]
Because $\bar q\ge 0$ is arbitrary and recalling the definition of $\underline R F_0$ in \eqref{defi:relax-op},
the previous inequality implies  that $u$ is a strong $\underline R F_0$-subsolution.
\bigskip

\noindent {\sc Strong implies weak.}
Assume that $u$ is a strong $\underline R F_0$-subsolution. Consider a test function $\varphi$ touching $u$ from above at $P_0=(t_0,x_0)$ with $t_0>0$
and $x_0 \in \partial \Omega$. Then we have
\[\varphi_t(P_0) + \underline R F_0(t,x,p)\le 0\quad \text{with}\quad p:=D \varphi(P_0).\]
Because $\underline R F_0 \ge (F_0 \wedge H)$, we deduce that
\[\varphi_t(P_0) + (F_0 \wedge H)(t,x,p)\le 0\]
which shows that  $u$ is a weak $F_0$-subsolution.
\end{proof}

Even if Lemma \ref{lem:relax-op7} gives a full characterization of weak solutions in terms of strong solutions, it is not completely satisfactory, because we may have $\underline R F_0< \overline R F_0$, and we would like to have the same boundary function. This is achieved in the following two results (for subsolutions and for supersolutions) where the common boundary function is $\mathfrak R F_0$.
\begin{pro}[Weak $F_0$-subsolutions are strong $\mathfrak R F_0$-subsolutions]
  \label{pro::28}
Assume that $H,F_0$ satisfy \eqref{a:HandF0}.  Consider an upper semi-continuous function $u: Q \to \R$.
\begin{enumerate}[label=(\roman*)]
\item 
If $u$ is a weak $F_0$-subsolution  of \eqref{eq:HJ-bc} and if for all $t>0$ and $x_0 \in \partial \Omega$,
\begin{equation}\label{eq:weak-cont}
u^*(t,x_0)=\limsup_{(s,y)\to (t,x_0), y \in \Omega} u(s,y)
\end{equation}
then $u$ is a strong $\mathfrak R F_0$-subsolution  of \eqref{eq:HJ-bc}.
\item
If $u$ is a strong $\mathfrak R F_0$-subsolution  of \eqref{eq:HJ-bc}, then $u$ is a weak $F_0$-subsolution  of \eqref{eq:HJ-bc}.
\end{enumerate}
\end{pro}
\begin{proof}
Let \(F:=\mathfrak R F_0.\)

Let $u$ be a weak $F_0$-subsolution of \eqref{eq:HJ-bc} satisfying the weak continuity condition \eqref{eq:weak-cont}. 
Consider a test function $\varphi$ touching $u$ from above at $P_0=(t_0,x_0)$ with $t_0>0$ and $x_0 \in \partial \Omega$.
Setting $p:=D\varphi(P_0)$ and $\lambda:=-\varphi_t(P_0)$, we have
\[(F_0 \wedge H)(t_0,x_0,p)\le \lambda.\]
Since we have $F=\underline R F$ (see Lemma~\ref{lem:relax-op}),
we know from  Proposition~\ref{pro:reduced} that we can assume that $p=(p',p_d)$ where $p_d$  is a negative characteristic point of $q_d \mapsto F(t_0,x_0,p'- q_d n(x_0))$.
From Corollary \ref{cor::30}, we deduce that
\(H(t_0,x_0,p)=F(t_0,x_0,p) \le F_0 (t_0,x_0,p) \)
and then
\(F(t_0,x_0,p)\le \lambda\)
which shows that $u$ is a strong $F$-subsolution.

If we assume now that $u$ is a strong $F$-subsolution, because $F=\overline R \underline R F_0 \ge \underline R F_0$,
we deduce that $u$ is also a strong $\underline R F_0$-subsolution. Then \ref{2romain} of Lemma \ref{lem:relax-op7} shows that $u$ is a weak $F_0$-subsolution. 
\end{proof}

Similarly, we show the following result.

\begin{pro}[Weak $F_0$-supersolutions are strong $\mathfrak R F_0$-supersolutions]
  \label{pro::28b}
Assume that $H,F_0$ satisfy (\ref{a:HandF0}).   Consider a lower semi-continuous function $u: Q \to \R$.
Then $u$ is a weak $F_0$-supersolution of \eqref{eq:HJ-bc} if and only if  $u$ is a strong $\mathfrak R F_0$-supersolution of \eqref{eq:HJ-bc}. 
\end{pro}

As a corollary of Lemma \ref{lem:weak-cont}, and of Propositions \ref{pro::28}, \ref{pro::28b}, we get the following equivalence
between weak $F_0$-solutions and strong $\mathfrak R F_0$-solutions.
\begin{theo}[Weak $F_0$-solutions are strong $\mathfrak R F_0$-solutions]
  \label{th:weak-strong}
Assume that $H,F_0$ satisfy (\ref{a:HandF0}).    
Assume that one of the following two conditions is satisfied:
\begin{enumerate}[label=(\roman*)]
\item \label{one}
either $F_0$ satisfies  the semi-coercivity condition (\ref{eq:semi-coercive}), 
\item or $u$ is weakly continuous at the boundary $\partial \Omega$, \textit{i.e.} it satisfies \eqref{eq:weak-cont}.
\end{enumerate}
Then a function $u: Q\to \R$ is a weak $F_0$-solution if and only if $u$ is a strong $\mathfrak R F_0$-solution.
\end{theo}
\begin{rem}
  This result under assumption~\ref{one} is exactly the same result as in  \cite[Theorem 1.3]{zbMATH06731793}, when we use our identification result Theorem \ref{th::GR}.
\end{rem}

\begin{counterexample}\label{rem::32}
  When we have neither the semi-coercivity of $F_0$, nor the weak continuity of the solution $u$, then $u$ can be a weak $F_0$-solution
  without being a strong $F$-solution for $F:=\mathfrak R F_0$, as shows the following counter-example.
We consider $\Omega = (0,+\infty)$ and 
\[ H(p):=|p|,\quad F_0\equiv 0,\quad F(p)=(\mathfrak R F_0)(p)=\max(-p,0) \]
where $F$ is semi-coercive, and for all $t>0$, we consider
\[ u(t,x)=\begin{cases}
1&\quad \text{if}\quad x=0\\
0 &\quad \text{if}\quad x>0.
\end{cases}
\]
One can check that $u$ is a (discontinuous) weak $F_0$-solution, but is not a strong $\mathfrak R F_0$-solution, neither a weak $\mathfrak R F_0$-solution.

On the contrary, for instance the function
\[ v(t,x)=\begin{cases}
-1&\quad \text{if}\quad x=0,\\
0 &\quad \text{if}\quad x>0
\end{cases}
\]
is both a (discontinuous) weak $F_0$-solution, and a strong $\mathfrak R F_0$-solution (and then also a weak $\mathfrak R F_0$-solution).
\end{counterexample}

\subsection{Existence and stability of weak solutions}

Given $T>0$, we consider the following problem,
\begin{equation}\label{eq::r95}
\begin{cases}
u_t + H(t,x, Du)=0&\quad \text{in}\quad (0,T) \times \Omega\\
u_t + F_0(t,x,Du)=0 &\quad \text{on}\quad (0,T)\times \partial \Omega
\end{cases}
\end{equation}
supplemented with the following initial condition
\begin{equation}\label{eq::r96}
u(0,\cdot)=u_0 \quad \text{in}\quad \left\{0\right\} \times \Omega.
\end{equation}

We have the following results. Their proofs are standard, so we skip it.
\begin{pro}[Stability of weak solutions by infimum/suppremum]\label{pro::r94b}
Assume that $H,F_0$ satisfy (\ref{a:HandF0}).
Let $\mathcal A$ be a non-empty set and let $(u_a)_{a\in \mathcal A}$ be a family of weak $F_0$-subsolutions (resp. weak $F_0$-supersolutions) of (\ref{eq::r95}). Let us assume that
\[u:=\sup_{a\in \mathcal A} u_a \quad (\text{resp.}\quad u:=\inf_{a\in \mathcal A} u_a)\]
is locally bounded on $(0,T)\times \overline \Omega$. Then $u^*$ is a weak $F_0$-subsolution (resp. $u_*$ is weak $F_0$-supersolution) of (\ref{eq::r95}).
\end{pro}
%
%
\begin{pro}[Stability of weak solutions by half-relaxed limits]\label{pro::r94c}
Assume that $H,F_0$ satisfy (\ref{a:HandF0}).
Let $(u_\varepsilon)_{\varepsilon}$ be a family of weak $F_0$-subsolutions (resp. weak $F_0$-supersolutions) of (\ref{eq::r95}). Let us assume that
the half-relaxed limit
$$u:=\limsup_{\varepsilon\to 0}{}^{*} u^\varepsilon \quad (\text{resp.}\quad u:=\liminf_{\varepsilon\to 0}{}_{*} u^\varepsilon)$$
is locally bounded on $(0,T)\times \overline \Omega$. Then $u$ is a weak $F_0$-subsolution (resp. weak $F_0$-supersolution) of \eqref{eq::r95}.
\end{pro}
Finally, we have the following existence result.

\begin{theo}[Existence of weak solutions]\label{th::r94}
  Assume that $H,F_0$ satisfy \eqref{a:HandF0} and $\Omega$ is bounded and that the initial data $u_0: \overline \Omega \to \R$ is uniformly continuous.
Then there exists a function $u \colon [0,T) \times \overline \Omega \to \R$ that is a weak $F_0$-solution of (\ref{eq::r95})-(\ref{eq::r96}) satisfying for some constant $C_T>0$
\[|u(t,x)-u_0(x)|\le C_T\quad \text{for all}\quad (t,x)\in [0,T)\times \overline \Omega. \]
\end{theo}
\begin{rem}
  The boundedness of $\Omega$ can be removed if one assumes for instance that
  \[ \sup_{t \in (0,T), x \in \Omega, p \in B_R} |H(t,x,p)| + |F_0(t,x,p)| < +\infty \]
  for all  $R>0$. 
\end{rem}
Such a result is proved by using  Perron's method. We recall that this method was introduced for viscosity solution by H. Ishii in \cite{zbMATH04142546}).
Here we skip the proof since it is completely similar to the proof of \cite[Theorem~2.14]{zbMATH06713740}.

\section{Guerand's approach}\label{s4}

This section is devoted to the proof of Theorem~\ref{th::GR}. We first recall the definition of Guerand's relaxation operator. 

\subsection{Guerand's relaxation operator}

The definition of Guerand's relaxation operator relies on the notion of limiter points. 
We split the set of limiter points $A_{F_0}$ into two subsets $A_{F_0}^+$ and $A_{F_0}^-$.
\begin{defi}[Positive and negative limiter points]
  \label{defi::3}
  \begin{enumerate}[label=(\roman*)]  
\item 
  A real number $p$ is a \emph{positive limiter point} of $F_0$ if $p^+ > p$ and $H(p) \ge F_0(p)$ and for all $q \in \R$, 
  \[ H(p) > H(q)\ge F_0(q) \Rightarrow (q^-,q^+)\cap (p,p^+)=\emptyset.\]
The set of all positive limiter points is denoted by $A^+_{F_0}$.
\item 
  A real number $p$ is a \emph{negative limiter point} of $F_0$ if $p^- < p$ and $H(p) \le F_0(p)$ and for all $q \in \R$, 
\[ F_0(q)\ge H(q) > H(p)  \Rightarrow (q^-,q^+)\cap (p^-,p)=\emptyset .\]
The set of all negative limiter points is denoted by $A^-_{F_0}$.
\item The set of all positive and negative limiter points is denoted by $A_{F_0}$.
\end{enumerate}
\end{defi}
\begin{rem}
  Remark that $A_{F_0}= \bigcup_{\alpha\in I} \left\{p_\alpha\right\}$ where $I$ is at most countable. Moreover, open intervals $(p_\alpha^-,p_\alpha^+)$ are disjoint, see
  \cite[Lemma~3.7]{zbMATH06731793}.
\end{rem}
\begin{defi}[Guerand's relaxation operator] \label{defi:guerand-relaxed}
We set for $p\in \R$
\[
  (\mathfrak JF_0)(p):=
  \begin{cases}
    H(p_\alpha) & \text{ if }  p\in [p_\alpha^-,p_\alpha^+] \text{ for some } p_\alpha \in A_{F_0},\\
    H(p) & \text{ elsewhere}.
  \end{cases}
\]
\end{defi}
\begin{rem}
In \cite{zbMATH06731793}, $\mathfrak JF_0$ is denoted by $F_{A_{F_0}}$.
\end{rem}
\begin{pro}[Property of $\mathfrak JF_0$, \cite{zbMATH06731793}] \label{pro::5}
The function $\mathfrak JF_0$ is well-defined, continuous and non-increasing.
\end{pro}

\subsection{Relaxation operators coincide}

In order to prove that $\mathfrak R F_0$ and $\mathfrak J F_0$ coincide,
we first prove that it is the case for limiter and characteristic points. 
\begin{pro}[Limiter points coincide with characteristic points of the relaxed function]\label{p:lim-char}
We have \(\chi^\pm (\mathfrak R F_0) = A^\pm_{F_0}.\) In other words, the characteristic points of the relaxed function coincide with the limiter points of the original function.
\end{pro}
\begin{proof}
We only do the proof for negative characteristic points since the proof for positive ones is very similar. Let $F=\mathfrak RF_0$.
\smallskip

\noindent {\sc Step 1: negative characteristic points are negative limiter points.}
Let \(p\in \chi^- (F).\) We have in particular $p^- < p$ and $H(p) =F (p)$. 
Then Corollary~\ref{cor::11} implies that
\begin{equation}\label{eq::12}
F=\mathfrak R F_0=\mathrm{constant}=H(p) \text{ in }  [p^-,p].
\end{equation}
We argue by contradiction and assume that $p\not\in A^-_{F_0}$. This means that there exists some $q\in \R$ such that
\begin{equation}\label{eq::5}
F_0(q)\ge H(q) > H(p)\quad \text{and} \quad (q^-,q^+)\cap (p^-,p)\not=\emptyset.
\end{equation}
Then \eqref{eq::5}  and \eqref{eq::1} imply in particular
\begin{equation}\label{eq::3}
F(q)\ge H(q) > H(p) =F(p)=F(p^-).
\end{equation}
This implies in particular that
\(q<p^-.\)

We next prove that  $p> q^+$. In order to do so, we first justify the fact that  $p\not\in [q,q^+]$.
Assume by contradiction that $p\in [q,q^+]$. Then this implies $H(p)\ge H(q)$, which contradicts \eqref{eq::3}. 
Then $p \not\in [q,q^+]$. If $p \le q$ then by monotonicity we have \(F(p)\ge F(q)\)
that contradicts \eqref{eq::3}. Hence  $p> q^+$. 

We deduce from \eqref{eq::5} and $q < p^-$ and $p > q^+$ that 
\[q^-\le q<p^-\le q^+<p.\]
This implies that \(H(p)=H(p^-)> H(q)=H(q^+)\), but  this is in contradiction with \eqref{eq::3}.
Hence $p \in A^-_{F_0}$. 
\smallskip

\noindent {\sc Step 2: negative limiter points are negative characteristic points.}
For $p\in A^-_{F_0}$, we have,
\begin{eqnarray}
  \nonumber
  && p^- < p \\
  \label{eq::13}
  && H< H(p) \le F_0(p) \le F_0 \text{ in } (p^-,p)\\
  \nonumber
  &&F(p)= \mathfrak R F_0(p) = \underline R F_0 (p) \ge H(p).
\end{eqnarray}
From \ref{l:loc} of Lemma \ref{lem:opti-loc}, we know that there exists $q\ge p$ minimal such that
\begin{equation}\label{eq::6}
\underline  R F_0(p)=(F_0 \wedge H)(q)
\end{equation}
\[\text{ with } \qquad
  \begin{cases}
    F_0(q)\ge H(q)&\\
    \underline R F_0= \mathrm{constant} = H(q) &\text{ in } [p,q],\\
    H(q)>H&\text{ in } [p,q).
  \end{cases}
\]
Hence by monotonicity of $F_0$, we have $F_0\ge H$ on $[p,q]$, and then
\[F=\underline RF_0=\mathrm{constant} =H(q)> H\quad \text{on}\quad [p,q).\]
Combined with \eqref{eq::13}, this implies
\[H<H(q) \quad \text{on}\quad (p^-,q)\quad \text{with}\quad p^-<p\le q.\]
We can now consider the lower point $q^-$ associated with $q$. We deduce from the previous inequality that
\[ q^-\le p^-< p\le q.\]
In particular \((q^-,q^+)\cap (p^-,p)\not= \emptyset.\)

If $q>p$, then we have \(F_0(q)\ge H(q) > H(p)\), in contradiction with the fact that $p\in A^-_{F_0}$.

We thus conclude that $q=p$, then \eqref{eq::6} shows that \(F(p)=H(p)\). Combined with \eqref{eq::13}, this yields
$p\in \chi^-(F)$. 
\end{proof}
We can now state and prove that the two relaxation operators are in fact the
same one. 
\begin{theo}[Relaxation operators coincide] \label{t:relax-coincide}
We assume that $H$ is continuous and coercive, and that $F_0$ is continuous, nonincreasing, and semi-coercive.
Then $\mathfrak R F_0=\mathfrak JF_0$.
\end{theo}
\begin{proof}
  We set $E = E^- \cup E^+$ with
\[E^-:=\bigcup_{\alpha\in I} [p_\alpha^-,p_\alpha] \quad \text{ and } \quad E^+:=\bigcup_{\alpha\in I} [p_\alpha,p_\alpha^+]\]
where $I$ is an at most countable set (see Proposition \ref{pro::5}) such that
\[A_{F_0}=\bigcup_{\alpha\in I} \left\{p_\alpha\right\}\]
with $p_\alpha^-\le p_\alpha \le p_\alpha^+$ and $p_\alpha^-< p_\alpha^+$.
We also set $F:=\mathfrak R F_0$.
\smallskip

\noindent {\sc Step 1: relaxation operators coincide in $E$.}
We only prove the result in $E^-$ since it can be obtained in $E^+$ similarly.
In the case where $p_\alpha\in A^-_{F_0}$, Proposition~\ref{p:lim-char} implies that $p_\alpha\in \chi^- (F)$, that is to say $p_\alpha^-< p_\alpha$ and, using also Corollary \ref{cor::11},
\[
H< F(p_\alpha)=H(p_\alpha)=H(p_\alpha^-) =\mathfrak JF_0 \text{ in } (p_\alpha^-,p_\alpha).
\]
Since $F$ is non-increasing, we have
\[\mathfrak R F_0=F> H \text{ in }  (p_\alpha^-,p_\alpha).\]
This implies that
\begin{equation}\label{eq::15}
  H<F=\underline R F_0 \le F_0 \text{ in } (p_\alpha^-,p_\alpha).
\end{equation}
Thanks to the continuity of $H$, $F$ and $\underline R F_0$, we
deduce from \eqref{eq::15} that
\[ H(p_\alpha) = F (p_\alpha) = \underline R F_0 (p_\alpha).\]
Hence 
\begin{align*}
  F(p_\alpha^-)& =\underline R F_0(p_\alpha^-)\\
               &=\max\left\{\sup_{q'\in [p_\alpha^-,p_\alpha)} (F_0 \wedge H)(q'),\ \sup_{q'\ge p_\alpha} (F_0 \wedge H)(q')\right\}\\
               & \le \max(H(p_\alpha),\underline RF_0(p_\alpha))\\
               & =F(p_\alpha).
\end{align*}
 From the monotonicity of $F$, we deduce that
\[ F=\mathrm{constant} =H(p_\alpha)= \mathfrak JF_0 \text{ in } [p_\alpha^-,p_\alpha].\]

\noindent {\sc Step 2: $\{ F \neq H\}$ is contained in $E$.}
If $p\in \left\{F\not= H\right\}$, then we know from Corollary~\ref{cor::14} that there exists $\varepsilon>0$ such that
\[F=\mathrm{constant} \text{ in }  (p-\varepsilon,p+\varepsilon).\]
We can then consider the largest interval $(a,b)\ni p$ such that
\[(a,b)\subset \left\{F\not= H\right\}.\]
Then $F$ is constant in $(a,b)$. The fact that $F$ is semi-coercive implies that $a > - \infty$. 
We distinguish two cases. 

If $F(p)> H(p)$, then from the coercivity of $H$ and the monotonicity of $F$, we have $a,b\in \R$ and
\[H(a)=F(a) = F(p)=F(b)=H(b)> H \text{ in } (a,b).\]
This implies that $b\in \chi^-(F)=A^-_{F_0}$ and $a:=b^-$ and in turn
$(a,b)\subset E^-$. In particular, $p \in E^-$ in this case. 

If $F(p)< H(p)$, we can then argue as in the previous case and get, thanks to Proposition~\ref{p:lim-char}, that
\[a\in \chi^+(F)=A^+_{F_0},\quad b=a^+\in \R\cup \left\{+\infty\right\}\]
and thus $(a,b)\subset E^+$. In particular, $p \in E^+$ in this case. 
\medskip

\noindent {\sc Step 3: conclusion.}
We proved that $F=\mathfrak JF_0$  in $E$ and also that
$F=H$ outside $E$.
Since $\mathfrak JF_0=H$ outside $E$ too (by definition), we thus get $F = \mathfrak JF_0$ everywhere. 
\end{proof}

\section{Godunov fluxes}\label{s5}

\subsection{Definition of Godunov fluxes}

We still consider a coercive and continuous Hamiltonian $H$ and we recall the standard Godunov flux associated to $H$ defined by 
\[ G(q,p)=\begin{cases}
\displaystyle    \max_{[p,q]} H & \text{ if } p\le q,\\
\displaystyle     \min_{[q,p]} H & \text{ if } p\ge q.
  \end{cases}
\]
In particular, $G$ is non-decreasing in the first variable and non-increasing in the second one. Moreover, we have $G(p,p)=H(p)$.
We  define next the action of the Godunov flux on a semi-coercive, continuous and non-increasing function $F_0$.
 \begin{pro}[Godunov's operator]\label{pro:n1}
 Assume that $F_0$ is semi-coercive, continuous and non-increasing and that $H$ is continuous and coercive. Let $p\in \R$, then the following properties hold true.
 \begin{enumerate}[label=(\roman*)]
 \item \label{lambdaq}
   There exists at least one $q \in \R$ such that $F_0(q)=G(q,p)$. The common value is denoted by  $\lambda_q$.
 \item \label{lambda}
   The value $\lambda_q$ defined above is independent on $q$. We denote this unique value by $\lambda=\lambda(p)=:(F_0G)(p)$
 \end{enumerate}
\end{pro}
 \begin{proof}
We first prove \ref{lambdaq}. Given $p \in \R$, 
   the function $\phi(q)=F_0(q)-G(q,p)$ is continuous and non-increasing.
 On the one hand, if $q\le p$, then $G(q,p)\le H(p)$ and $\phi(q)\ge F_0(q)-H(p)$. Using that $F_0$ is semi-coercive, we deduce that 
 \[
 \lim_{q\to-\infty}\phi(q)=+\infty.
 \]
 On the other hand, if $q\ge p$, using that $F_0(q)\le F_0(p)<+\infty$, the fact that $G(q,p)=\max_{[p,q]}H\ge H(q)$ and the fact that $H$ is coercive, we deduce that
 \[
 \lim_{q\to+\infty}\phi(q)=-\infty.
 \]
Since $\phi$ is continuous and non-increasing, we deduce the existence of a $q$ such that $\phi(q)=0$, that is to say that
$F_0(q)=G(q,p).$
\medskip

We now turn to \ref{lambda}.
By contradiction, assume that there exist $q_1$ and $q_2$ such that 
\[\lambda_{q_1}=F_0(q_1)=G(q_1,p)>\lambda_{q_2}=F_0(q_2)=G(q_2,p).\]
Since $F_0$ is non-increasing, we deduce that $q_1<q_2$. Using that $G$ is non-decreasing in its first argument, we deduce that $G(q_1,p)\le G(q_2, p)$ which is a contradiction.
 \end{proof}
 
 The goal is now to prove that $\mathfrak R F_0=F_0 G$. More precisely, we have the following theorem. 
\begin{theo}[Relaxation operator coincide with Godunov's operator]\label{th:n1}
Assume that $F_0$ is semi-coercive, continuous and non-increasing and that $H$ is continuous and coercive. Then 
\[\mathfrak RF_0=F_0G.\]
\end{theo}

In order to prove this theorem, we need to introduce the Godunov semi-fluxes. This is done in the next section.
The proof of Theorem \ref{th:n1} is postponed until Subsection \ref{sec:proof-th:n1}.
 
\subsection{Godunov semi-fluxes}

We introduce the Godunov semi-fluxes, $\underline G$ and $\overline G$, which are set-valued applications defined by
\[ \underline G(q,p)=
  \begin{cases}
    \{-\infty\}& \text{ if } q<p, \medskip\\
    {[-\infty, H(p)]}&\text{ if }q=p, \medskip\\
   \displaystyle \left\{\max_{[p,q] }H \right\}&\text{ if } q>p
  \end{cases}
\]
and
\[
  \overline G(q,p)=
  \begin{cases}
    \displaystyle \left\{\min_{[q,p] }H \right\}&\text{ if }q<p, \medskip\\
    {[ H(p), +\infty]}&\text{ if }q=p, \medskip\\
    \{+\infty\} &\text{ if } q>p.
  \end{cases}
\]
As before, we can define the action of these semi-fluxes on non-increasing semi-coercive continuous functions. 
\begin{pro}[Lower Godunov operator $F_0 \underline G$] \label{p:sub-G}
  Assume that $F_0$ is semi-coercive, continuous and non-increasing and that $H$ is continuous and coercive. Let $p\in \R$.
  We define the sets
\[\underline Q:=\{q\in \R, F_0(q)\in \underline G(q,p)\} \quad \text{ and } \quad \underline \Lambda:=\{F_0(q), \; q\in \underline Q\}.\]
 Then the following properties hold true.
 \begin{enumerate}[label=(\roman*)]
 \item \label{Q} The set $\underline Q$ is non-empty and contained in $[p,+\infty[$.
 \item \label{Lambda} The set $\underline \Lambda$ is reduced to a singleton that we denote by $\{(F_0\underline G)(p)\}$. 
 \end{enumerate}
\end{pro}
\begin{proof}
We first prove \ref{Q}. In order to do so, we distinguish two cases.
  
Suppose first that $F_0(p)> H(p)$. In that case, we remark that $\underline G(q,p)=\{G(q,p)\}$ for all $q > p$. Then, the proof
is the same as the one of Proposition \ref{pro:n1}. Indeed, if we define $\phi(q)=F_0(q)-G(q,p)$, then $\phi(p)=F_0(p)-G(p,p)=F_0(p)-H(p)>0$
and so the zero of $\phi$ defined in the proof of Proposition \ref{pro:n1}  is greater than $p$ and satisfies the desired condition.

Suppose now that $F_0(p)\le H(p)$. In that case, we remark that $p\in \underline Q$ since $\underline G(p,p)=[-\infty, H(p)]$.

 The proof of \ref{Lambda} follow the same lines as the one of \ref{lambda} from Proposition \ref{pro:n1}.
\end{proof}

In the same way, we have the following proposition concerning $\overline G$. Since the proof is similar to the
previous one, we skip it. 
\begin{pro}[Upper Godunov operator $F_0 \overline G$]
Assume that $F_0$ is semi-coercive, continuous and non-increasing and that $H$ is continuous and coercive. Let $p\in \R$. We define the sets
\[\overline Q:=\{q\in \R, F_0(q)\in \overline G(q,p)\} \quad \text{ and } \quad \overline \Lambda:=\{F_0(q), \; q\in \overline Q\}.\]
Then the following properties hold true
 \begin{enumerate}[label=(\roman*)]
 \item The set $\overline Q$ is non-empty and contained in $]-\infty, p]$.
 \item The set $\overline \Lambda$ is reduced to a singleton that we denote by $\{(F_0\overline G)(p)\}$. 
 \end{enumerate}
\end{pro}
In order to compose semi-Godunov operators, we first need to make sure that $F_0\underline G$ satisfy the same assumptions as $F_0$. 
\begin{lem}[Properties of $F_0\underline G$ and $F_0\overline G$]\label{lem:n4}
Under the same assumptions, $F_0\underline G$ and $F_0\overline G$  are non-increasing, continuous and semi-coercive.
\end{lem}
\begin{proof}
We do the proof only for $F_0\underline G$, the one for $F_0 \overline G$ being similar. 

We first show that $F_0 \underline G$ is non-increasing. Let $p_1> p_2$ and $q_1, q_2$ be such that $(F_0\underline G)(p_i)=F_0(q_i)\in \underline G(q_i,p_i)$ for $i\in \{1,2\}$. In particular,
since $q_i \in \underline Q$, we have $q_i\ge p_i$ thanks to \ref{Q} from Proposition~\ref{p:sub-G}. 

We assume by contradiction that 
$(F_0\underline G)(p_1)>(F_0\underline G)(p_2)$. This implies $F_0 (q_1) > F_0 (q_2)$ and in particular  $q_2> q_1\ge p_1>p_2$. Hence $\underline G(q_2,p_2)=\{G(q_2,p_2)\}$ and so $F_0(q_2)= G(q_2,p_2)\ge G(q_1,p_1)\ge G(p_1,p_1)$. The inequalities follow from monotonicity properties of $G$ in both variables. If $q_1>p_1$, then  $F_0(q_1)=G(q_1,p_1)$ and we get a contradiction: $F_0(q_1) \le F_0 (q_2)$. If $q_1=p_1$, then $\underline G (p_1,q_1) = [-\infty, H(p_1)]$ from which we get $F_0(q_1)\le H(p_1)=G(p_1,p_1) \le F_0 (q_2)$ and we get the same contradiction.
\medskip

We now prove that $F_0\underline G$ is semi-coercive. Let $M>0$. There exists $p_0$ such that for every $p<p_0$, $H(p)\ge M$ and $F_0(p)\ge M$.
Let $p<p_0$. Proposition~\ref{p:sub-G} implies that there exists $q\ge p$  such that $(F_0\underline G)(p)=F_0(q)\in \underline G(q,p).$ If $q=p$, then $(F_0\underline G)(p)=F_0(p)\ge M$. If $q>p$, then $(F_0\underline G)(p)=G(q,p)\ge H(p)\ge M$. This shows that $F_0\underline G$ is semi-coercive.
\medskip

We now prove that $F_0\underline G$ is continuous. Let $p_n\to p$ and $q_n \ge p_n$ be such that $(F_0\underline G)(p_n)=F_0(q_n)\in \underline G(q_n,p_n)$.
From the coercivity of $H$, we get that $(q_n)_n$ is bounded: indeed, either $q_n=p_n$ or $F_0(p_n) \ge F_0(q_n) = G(q_n,p_n) \ge H (q_n)$. 
Hence, up to extract a subsequence (still denoted by $(q_n)_n$), we have $q_n \to q_0\ge p$.

Assume first that $q_{n_j}=p_{n_j}$ along a subsequence $\{n_j\}$. In this case $F_0(p_{n_j})=F_0(q_{n_j})\le H(p_{n_j})$. This implies that $F_0(p)\le H(p)$ and so $F_0(p)\in \underline G(p,p)$.
This means that $p \in \underline Q$ and $F_0 (p) = (F_0 \underline G)(p)$ and
\[(F_0\underline G)(p_{n_j})=F_0(q_{n_j})\to F_0(p)=(F_0\underline G)(p). \]

Assume now that $q_n > p_n$ for $n$ large enough, then $(F_0\underline G)(p_n)= F_0(q_n) = G (q_n,p_n)$. 
Since $q_n\to q_0\ge p$ and $F_0(q_{n}) \le F_0 (p_{n})$, we get $F_0(q_0)\le F_0(p)$ and 
$F_0 (q_0)= G(q_0,p)$.

If $q_0=p$ then $F_0(p) = H(p) \in \underline G (p,p)$. If $q_0 >p$ then  $F_0(q_0)\in \underline G(q_0,p)$.
In both cases, $q_0 \in \underline Q$ and thus $F_0(q_0) = (F_0 \underline G) (p)$. We thus proved that 
$F_0\underline G(p_n)=F_0(q_n)\to F_0(q_0) = F_0\underline G(p)$.
This implies that indeed the whole sequence $\{(F_0 \underline G)(p_n)\}$ converges to $(F_0 \underline G) (p)$.
\end{proof}

We now want to prove that the action of $\underline G$ on the action of $\overline G$ on $F_0$ is in fact the action of $G$ on $F_0$. 
\begin{pro}[Composition of Godunov semi-fluxes]\label{pro:n3}
We have $(F_0\overline G)\underline G =F_0G=(F_0\underline G)\overline G.$
\end{pro}
In order to prove this proposition,  the following lemma is needed.
\begin{lem}[Key composition result]\label{lem:n1}
  \begin{enumerate}[label=(\roman*)]
  \item \label{first}
    For all $(q,p) \in \R^2$, there exists $q' \in \R$ such that $\overline G (q,q')\cap \underline G(q',p)\ne \emptyset$. Moreover, for such a real number $q'$, we have $\overline G (q,q')\cap \underline G(q',p)=\{G(q,p) \}$.
\item \label{second}
  For all $(q,p)$, there exists $q'\in \R$ such that $\underline G (q,q')\cap \overline G(q',p)\ne \emptyset$. Moreover, for such a real number $q'$, we have $\underline G (q,q')\cap \overline G(q',p)=\left\{G(q,p)\right\}$.
\end{enumerate}
\end{lem}
\begin{proof}
We only prove \ref{first} since the proof of \ref{second} follows the same reasoning.
  
 We first show that $\overline G (q,q')\cap \underline G(q',p)$ is either empty or equal to the singleton $\{G(q,p)\}$.

Remark that the intersection can only contain real numbers, but neither $+\infty$ nor $-\infty$. Hence, if the intersection
is not empty, then $p \le q'$ and $q \le q'$.  We now distinguish four cases.
\medskip

\noindent{\bf Case 1: $p=q=q'$.} In that case $\underline G(p,p)=[-\infty, H(p)]$ and $\overline G(p,p)=[H(p),+\infty]$ and so the intersection is reduced to a singleton of element $H(p)=G(p,p)=G(q,p)$.\medskip

\noindent{\bf Case 2: $p<q=q'$.} In that case $\overline G(q,q')=[H(q), +\infty]$ and $\underline G(q',p)=\{G(q',p)\}=\{G(q,p)\}$. Since $q\ge p$, we have $G(q,p)\ge G(q,q)=H(q)$ and so the intersection is non-empty and then reduced to $G(q,p)$.
\medskip

\noindent{\bf Case 3: $q<p=q'$.} In that case $\overline G(q,q')=\{G(q,p)\}$ and $\underline G(q',p)=[-\infty, H(p)]$. Since $q\le p$, we have $G(q,p)\le G(p,p)=H(p)$ and so the intersection is reduced to $G(q,p)$.
\medskip

\noindent{\bf Case 4: $q<q'$ and $p<q'$.} In that case $\overline G(q,q')=\{G(q,q')\}$ and $\underline G(q',p)=\{G(q',p)\}$. If the interscetion is not empty, then $G(q,q')=G(q',p)$, which means that 
$$\max_{[p,q']}H=\min_{[q,q']} H,$$ 
i.e. $H$ is constant on $[\max(q,p),q']$. If $p<q$, this implies in particular that 
$$G(q,p)=\max_{[p,q]} H=\max_{[p,q']}H=G(q',p).$$ 
Similarly if $p>q$, we get $G(q,p)=G(q,q')$. In the last case $p=q$, we get 
$$G(q,p)=G(q,q')=G(q',p).$$

\bigskip
We now prove that we can always find a $q'$ such that the intersection is non empty.  If $p=q$, we can take $q'=p=q$ as in Case 1. If $p<q$, we can take $q'=q$ as in Case 2, while if $p>q$, we can take $q'=p$ as in Case 3.
\end{proof}

We are now able to prove Proposition \ref{pro:n3}.
\begin{proof}[Proof of Proposition \ref{pro:n3}]
Let $F_1=F_0\overline G$. We use successively the definition of $F_0 G$, \ref{first} from Lemma~\ref{lem:n1}, the definitions of
$F_0 \overline G$ and of $F_1 \underline G$ to write,
\begin{align*}
  \{ F_0 G (p) \} &= \{ F_0 (q) \text{ for some $q$ s.t. } F_0(q) \in G(q,p)\} \\
& = \{ F_0 (q) \text{ for some $q$ and $q'$ s.t. } F_0 (q) \in \overline G (q,q') \cap \underline G(q',p)\} \\
  \{ F_1 (q') \} &= \{F_0 \overline G (q') \} = \{ F_0 (q) \text{ for some $q$ s.t. } F_0 (q) \in \overline G (q,q')\} \\
  \{ F_1 \underline G (p) \} &= \{ F_1 (q') \text{ for some $q'$ s.t. } F_1 (q') \in \underline G(q',p) \} \\
  & = \{ F_0 (q) \text{ for some $q$ and $q'$ s.t. } F_0 (q) \in \overline G (q,q') \cap \underline G(q',p)\}.
\end{align*}
This implies that $F_0 G (p) = F_1 \underline G (p)= (F_0 \overline G) \underline G$.

Using \ref{second} from Lemma~\ref{lem:n1}, we can follow the same reasoning and get $F_0 G (p) = (F_0 \underline G) \overline G$.
\end{proof}


\subsection{Relaxation and Godunov fluxes}
\label{sec:proof-th:n1}

The proof of Theorem \ref{th:n1} is a direct consequence of the following proposition which makes the link between the semi-relaxation
of $F_0$ and the actions of the Godunov semi-fluxes on $F_0$.
\begin{pro}[Semi-relaxations and Godunov's semi-fluxes]\label{pro:n2}
Assume that $F_0$ is semi-coercive, continuous and non-increasing and that $H$ is continuous and coercive. Then
\(F_0 \underline G=\underline R F_0\)  and \(F_0 \overline G=\overline R F_0.\)
\end{pro}
\begin{proof}
  We only prove that $F_0 \underline G=\underline R F_0$ since the proof of the other equality is similar.
  Let $p$ and $q'\ge p$ be such that
\[(F_0\underline G)(p)=F_0(q')\in \underline G(q',p).\]

If $q'=p$, then $(F_0\underline G)(p)= F_0(p)\le H(p)$. Using Lemma \ref{lem:relax-prop-1}, we deduce that 
\[\underline RF_0(p)=F_0(p)=F_0\underline G (p).\]

If $q'>p$, then $\displaystyle F_0(q')=G(q',p)=\max _{[p,q']}H$. In particular $F_0(q') \ge H(q')$
and by Lemma \ref{lem:relax-prop-1}, we have $\underline RF_0(q')\le F_0(q')$.
Recall also that
\[\underline RF_0(p)=\max\left(\sup_{[p,q']}(F_0 \wedge H), \underline RF_0(q')\right).\]
Since $F_0$ is non-increasing, we have for all $q\in[p,q']$, 
\[F_0(q)\ge F_0(q')=\max _{[p,q']}H.\]
In particular,
\[\sup_{q \in [p,q']} (F_0 \wedge H)(q) = \max_{q \in [p,q']} H(q) = F_0 (q') \ge \underline RF_0(q')\]
and we finally get
\[\underline RF_0(p)=F_0(q')=(F_0\underline G)(p). \qedhere\]
\end{proof}
We now turn to the proof of Theorem \ref{th:n1}.
\begin{proof}[Proof of Theorem \ref{th:n1}]
Lemma \ref{lem:n4} implies that $F_0 \overline G$ satisfies
  the assumptions of Proposition \ref{pro:n2}.  Using Proposition \ref{pro:n2} first to $F_0$ and then to $\hat{F}_0 = F_0 \overline G$, we have
\[\underline R (\overline R F_0)=\underline R (F_0 \overline G)= (F_0\overline G)\underline G.\]
Using Lemma \ref{lem:relax-op} and Proposition \ref{pro:n3}, we then get \(\mathfrak R F_0=F_0G. \)
\end{proof}

\section{The Neumann and Dirichlet problems}
\label{sec:neumann-dir}

\subsection{Strong solutions for the Neumann problem}

This subsection is devoted to the proof of Theorem \ref{t:neumann}. 

\begin{proof}[Proof of Theorem \ref{t:neumann}]
The proof is split in several steps. 
  
\noindent  \textsc{Step 1: the condition $N$ is self-relaxed.}
We recall that
\[
  N(t,x,p) =
  \begin{cases}
    \max \bigg\{ H(t,x,p- \rho n) : \rho \in [0,p \cdot n(x) + \rho_0] \bigg \} & \text{ if } p \cdot n(x) + \rho_0 \ge 0, \\
    \min \bigg\{ H(t,x,p- \rho n) : \rho \in [p \cdot n(x) + \rho_0,0] \bigg\} & \text{ if } p \cdot n(x) + \rho_0 \le 0
  \end{cases}
\]
with $\rho_0 = h(t,x)$. 
For $p=p'-\rho n$ with $p' \perp n$, it is convenient to consider $H_0 (\rho) = H (t,x,p'-\rho n)$ and  $N_0 \colon \rho \mapsto N (t,x,p'- \rho n)$.
In particular,
\[ N_0 (\rho) = \begin{cases}
\displaystyle    \min_{[\rho_0,\rho]} H_0  & \text{ if }  \rho \ge \rho_0, \\
\displaystyle     \max_{[\rho,\rho_0]} H_0  & \text{ if } \rho \le \rho_0.
  \end{cases}
\]
In other words, $N_0(\rho) = G(\rho_0,\rho)$ where $G$ denotes the Godunov flux function. 
We remark that $N_0$ is \emph{self-relaxed} in the sense that $\mathfrak R N_0 = N_0$. 
Indeed, we remark that
\[ ( H_0 (\rho)-N_0 (\rho) ) ( \rho - \rho_0) \ge 0.\]
In particular, thanks to Lemma~\ref{lem:relax-prop-1}, we know that $\underline R N_0 = N_0$ in $(\rho_0,+\infty) \subset \{ N_0 \le H_0 \}$. 
For $\rho \le \rho_0$, we write
\begin{align*}
  \underline R N_0 (\rho) &= \max_{q \ge \rho} (N_0 \wedge H_0) (q) \\
                             & = \bigg( \max_{q \in [\rho,\rho_0]} H_0 (q) \bigg) \vee \underline R N_0 (\rho_0) \\
                             & = N_0 (\rho) \vee N_0 (\rho_0) \\
  &= N_0 (\rho). 
\end{align*}
Hence $\underline R N_0 = N_0$ in $\R$. Similarly, $\overline R N_0 = N_0$ and $\mathfrak R N_0 = N_0$. 
\medskip

We observe next that negative characteristic points of $N_0$ are contained in $(-\infty,\rho_0]$.
Indeed, if $\rho > \rho_0$ and $N_0 (\rho) = H_0(\rho)$, then $N_0 (\rho) = \min_{[\rho_0,\rho]} H_0$ and in particular,
$H_0 \ge N_0 (\rho) = H_0 (\rho)$ in $[\rho_0,\rho]$. In particular, $\rho$ is not a negative characteristic point of $N_0$.
\medskip

\noindent \textsc{Step 2: weak solutions of the Neumann problem are strong $N$-solutions.}
We only treat the case of weak subsolutions since weak supersolutions can be treated similarly.

Let $u \colon Q \to \R$ be a weak solution of \eqref{eq:HJ-neu}. Then Lemma~\ref{lem:weak-cont} implies that $u$ is weakly continuous.

Thanks to Proposition~\ref{pro:reduced}, we only consider a $C^1$ test function  $\phi$ touching $u^*$  from above at $P_0=(t_0,x_0)$ with $x_0 \in \partial \Omega$
of the form
\[ \phi (t,x) = \psi (t,x') + \rho x_d \]
for a negative characteristic point $\rho$ of $N_0$ (recall that $\underline R N_0=N_0$). In particular, $\rho \le \rho_0$.

Consider $r>0$ and $\gamma \in C^1 (\R^{d-1})$ such that \eqref{e:pomega} holds true. 
Then we have the viscosity inequality,
\[
  \phi_t + H(t,x,D \phi) = \min\left(\phi_t + H(t,x,D \phi) ,\rho_0 -\rho \right)\le 0\quad \text{at}\quad P_0.
\]
For $p=D \phi(P_0)$  and  $R\in [0,\rho_0-\rho]$,
the function $\varphi(t,x)=\phi(t,x)+R(x_d - \gamma (x'))$ is still a test function for $u$ at $P_0$. Since
$D' \gamma (x'_0) =0$ and $R + \rho \le \rho_0$, 
\[ \phi_t (t_0,x_0)+ \max_{R\in [0, \rho_0-\rho]}  H(t_0,x_0,D \phi (t_0,x_0) - R n (x_0)) \le 0.\]
Since $\rho = -\frac{\partial \phi}{\partial n} (t_0,x_0)$, this precisely means $\phi_t + N (t_0,x_0,D \phi) \le 0$ at $P_0$. 
\medskip

\noindent \textsc{Step 3: strong $N$-solutions are weak solutions of the Neumann problem.}
We show it for strong $N$-subsolutions since the proof for strong $N$-supersolutions is similar.
Assume that $u$ is a strong $N$-subsolution. Let $\varphi$ be a $C^1$ test function touching $u^*$
from above at $P_0=(t_0,0)$. Letting $\lambda:=\varphi_t(P_0)$ and $p:=D \varphi(P_0)$, we have
\[\lambda+ N(t_0,x_0,p)\le 0.\]
If $\rho= - p \cdot n(x_0) \le \rho_0 = h (t_0,x_0)$, then  $N(t_0,x_0,p)=N_0 (\rho) \ge H_0(\rho) = H (t_0,x_0,p)$, which implies
\[\lambda+ H(t_0,x_0,p)\le 0.\]
In particular,
\[\min(\lambda+H(t_0,x_0,p),h(t_0,x_0) + p \cdot n(x_0))\le 0.\]
If $\rho= -p \cdot n(x_0) > \rho_0 = h (t_0,x_0)$, the previous inequality also holds true. 
\end{proof}

\subsection{Connection with scalar conservation laws}\label{ss.4}

In this subsection, we would like to make a link between the relaxation operator $\mathfrak{R} F_0$ and the theory of boundary conditions for scalar conservation laws\footnotemark[2].
\footnotetext[2]{Morally if $u$ is a strong $\mathfrak R F_0$-solution that is Lipschitz continuous, then the function $v:=u_x$ is expected to be an entropy solution of
$$\left\{\begin{array}{lll}
v_t + H(v)_x=0, & \text{on}&\quad (0,+\infty)_t\times (0,+\infty)_x,\\
v(t,0) \in \mathcal G, & \text{for a.e.}&\quad t\in (0,+\infty)
\end{array}\right.$$
where $v(t,0)$ is a strong (quasi)-trace of $v$ in the sense of Panov \cite{zbMATH05275060}. It is possible to prove it, if $H$ is $C^1$ and $H'$ is not constant on every interval of positive length. But it requires some additional work which is out of the scope of the present paper.
}
 To this end, we consider a linear function,
$$u(t,x)=px+ \lambda t.$$
It is straightforward to check that it is a weak viscosity solution of (\ref{eq:HJ-bc}) if and only if $\lambda=-H(p)$ and
\begin{equation}\label{eq::r2}
(\underline R F_0)(p) \le H(p)=-\lambda  \le (\overline RF_0)(p).
\end{equation}
Then we have
\begin{lem}[Relation with the germ]\label{lem::r3}
Assume  (\ref{eq:HJ-bc}). An element $p\in \R$ satisfies (\ref{eq::r2}) if and only if $p$ is an element of the set (which is called a germ)
\begin{equation}\label{eq::r4}
\mathcal G=\left\{q\in \R,\quad H(q)=\mathfrak R F_0(q)\right\}.
\end{equation}
\end{lem}
\begin{rem}
For the notion of germ and its properties (maximal germs, complete germs) we refer the reader to \cite{zbMATH06101925}. 
\end{rem}
\begin{rem}
The fact that $\mathfrak R F_0$ is nonincreasing provides to the set $\mathcal G$ the property to be a germ for $H$.
Moreover  it is possible to check that this germ is maximal if and only if it is of the form of (\ref{eq::r4}) for some suitable $F_0$.
With some further work, it is also possible to show that the germ $\mathcal G$ is complete  for instance if $H\in C^1$ (but it is out of the scope of this paper).
\end{rem}

\begin{proof}Recall from Lemma \ref{lem:relax-prop-1} and Remark \ref{rem:relax-prop-1} that 
$$\left\{\begin{array}{l}
(\overline R F_0)(p)\ \left\{\begin{array}{ll}
=F_0(p)&\quad \text{if}\quad F_0(p)\ge H(p)\\
\in [F_0(p),H(p)]&\quad \text{if}\quad F_0(p)\le H(p)\\
\end{array}\right.\\
\\
(\underline R F_0)(p)\ \left\{\begin{array}{ll}
\in [H(p),F_0(p)]&\quad \text{if}\quad F_0(p)\ge H(p)\\
=F_0(p)&\quad \text{if}\quad F_0(p)\le H(p)\\
\end{array}\right.\\
\\
\underline R F_0\le \overline RF_0\\
\\
(\overline R F_0)(p) =H(p)=(\underline R F_0)(p) \quad \text{if}\quad F_0(p)=H(p).\\
\end{array}\right.$$
Hence we deduce that (\ref{eq::r2}) is equivalent to
$$-\lambda=H(p)=\left\{\begin{array}{ll}
(\overline R F_0)(p) &\quad \text{if}\quad F_0(p)\le H(p)\\
(\underline R F_0)(p) &\quad \text{if}\quad F_0(p)\ge H(p)\\
\end{array}\right\} =(\mathfrak R F_0)(p)$$
\textit{i.e.} \(p\in \mathcal G:=\left\{H=\mathfrak R F_0\right\}\)
which ends the proof of the lemma.
\end{proof}

\subsection{Strong solutions for the Dirichlet problem}

In this subsection, we compute the relaxed Dirichlet boundary condition. 

\begin{proof}[Proof of Theorem~\ref{t:dirichlet}]
Let $u\colon Q \to \R$ be a weak viscosity subsolution of \eqref{eq:HJ-dir}.
Let $\phi$ be a $C^1$ test function touching $u^*$ from above at $P_0 = (t_0,x_0)$ with
$x_0 \in \partial \Omega$. Then we have
\[ \min (u^*-g, \phi_t + H(t,x,D \phi)) \le 0 \quad \text{ at } P_0.\]
This boundary condition can be interpreted as follows,
\[ \phi_t + \min (u^*-g - \lambda, H (t,x,D \phi) \le 0 \quad \text{ at } P_0 \]
where $\lambda = \phi_t (t_0,x_0)$ (recall that we look at pointwise inequality and that only the behavior in the normal gradient is taken into account).
We can argue similarly for weak viscosity supersolutions of \eqref{eq:HJ-dir}
and we conclude that the Dirichlet condition can be interpreted
as a dynamic boundary condition with $F_0 (p) = u^*(t_0,x_0)- g(t_0,x_0) - \lambda =: A_0$. 

Recalling the definition of $\underline RF_0$ and $\overline RF_0$, see \eqref{defi:relax-op}, we  compute, 
\begin{align*}
    \underline R F_0 (p) &= \sup_{\rho \ge 0} \min (A_0, H (t_0,x_0,p- \rho n (x_0)) \\
                         & =  \min (A_0, \sup_{\rho \ge 0} H (t_0,x_0,p- \rho n (x_0)) \\
  & = A_0 \\
  \overline R F_0 (p) &= \inf_{\rho \le 0} \max (A_0, H (t_0,x_0,p - \rho n (x_0)) \\
  & =  \max (A_0, \inf_{\rho \le 0} H (t_0,x_0,p - \rho n (x_0)) \\
& = \max (A_0, H_- (t_0,x_0,p)).
\end{align*}
Recalling now the definition of the relaxation operator, see \eqref{eq::r12},
\begin{align*}
  \mathfrak R F_0 (p) &=
  \begin{cases}
     A_0 & \text{ if } H(t_0,x_0,p) \le A_0, \\
    \max (A_0, H_- (t_0,x_0,p)) & \text{ if } H(t_0,x_0,p) \ge A_0 
  \end{cases} \\
  & = \max (A_0, H_- (t_0,x_0,p)).
\end{align*}
We used the fact that $H \ge H_-$ to get the last line. 
Recalling that $A_0 = u^*(t_0,x_0)- g(t_0,x_0) - \lambda$, Theorem~\ref{th:weak-strong} implies
the conclusion of Theorem~\ref{t:dirichlet}.
\end{proof}

\small

\paragraph{Acknowledgements} This research was funded, in whole or in part, by l'Agence Nationale de la Recherche (ANR), project ANR-22-CE40-0010. For the purpose of open access, the author has applied a CC-BY public copyright licence to any Author Accepted Manuscript (AAM) version arising from this submission.

\bibliographystyle{siam}
\bibliography{bib-relax}

\end{document}